\documentclass[11pt]{article}
\usepackage{amsmath,amsthm,amssymb}
\usepackage[usenames,dvipsnames]{xcolor}
\usepackage{enumerate}
\usepackage{graphicx}
\usepackage{cite}
\usepackage{comment}
\usepackage{oands}
\usepackage{tikz}
\usepackage{changepage}
\usepackage{bbm}
\usepackage{mathtools}
\usepackage[margin=1in]{geometry} 
\usepackage{authblk}
\usepackage[pagewise,mathlines]{lineno}
\usepackage{appendix}
\usepackage{multicol}
\usepackage{stmaryrd} 

\usepackage[pdftitle={Introduction to the Liouville quantum gravity metric}, pdfauthor={Jian Ding, Julien Dub\'edat, Ewain Gwynne}, colorlinks=true,linkcolor=NavyBlue,urlcolor=RoyalBlue,citecolor=PineGreen,bookmarks=true,bookmarksopen=true,bookmarksopenlevel=2,unicode=true,linktocpage]{hyperref}

\theoremstyle{plain}
\newtheorem{thm}{Theorem}[section]

\newtheorem{lem}[thm]{Lemma}
\newtheorem{prop}[thm]{Proposition}

\def\@rst #1 #2other{#1}
\newcommand\MR[1]{\relax\ifhmode\unskip\spacefactor3000 \space\fi
  \MRhref{\expandafter\@rst #1 other}{#1}}
\newcommand{\MRhref}[2]{\href{http://www.ams.org/mathscinet-getitem?mr=#1}{MR#2}}

\theoremstyle{definition}
\newtheorem{defn}[thm]{Definition}
\newtheorem{remark}[thm]{Remark}
\newtheorem{prob}[thm]{Problem}

\numberwithin{equation}{section}

\newcommand{\dsb}{\begin{adjustwidth}{2.5em}{0pt}
\begin{footnotesize}}
\newcommand{\dse}{\end{footnotesize}
\end{adjustwidth}}

\newcommand{\ssb}{\begin{adjustwidth}{2.5em}{0pt}}
\newcommand{\sse}{\end{adjustwidth}}

\newcommand{\aryb}{\begin{eqnarray*}}
\newcommand{\arye}{\end{eqnarray*}}
\def\alb#1\ale{\begin{align*}#1\end{align*}}
\def\allb#1\alle{\begin{align}#1\end{align}}
\newcommand{\eqb}{\begin{equation}}
\newcommand{\eqe}{\end{equation}}
\newcommand{\eqbn}{\begin{equation*}}
\newcommand{\eqen}{\end{equation*}}

\newcommand{\BB}{\mathbbm}
\newcommand{\ol}{\overline}

\newcommand{\op}{\operatorname}

\newcommand{\frk}{\mathfrak}
\newcommand{\eqD}{\overset{d}{=}}
\newcommand{\ep}{\varepsilon}
\newcommand{\rta}{\rightarrow}

\newcommand{\wt}{\widetilde}
\newcommand{\wh}{\widehat} 
\newcommand{\mcl}{\mathcal}

\newcommand{\bdy}{\partial}

\newcommand{\ccM}{{\mathbf{c}_{\mathrm M}}}
\newcommand{\tr}{{\mathrm{tr}}}
   
\let\originalleft\left
\let\originalright\right
\renewcommand{\left}{\mathopen{}\mathclose\bgroup\originalleft}
\renewcommand{\right}{\aftergroup\egroup\originalright}

\title{Introduction to the Liouville quantum gravity metric}
\date{    }
\author{Jian Ding, Julien Dub\'edat, Ewain Gwynne}

\begin{document}

\maketitle

\begin{abstract}
Liouville quantum gravity (LQG) is a one-parameter family of models of random fractal surfaces which first appeared in the physics literature in the 1980s. Recent works have constructed a metric (distance function) on an LQG surface. We give an overview of the construction of this metric and discuss some of its most important properties, such as the behavior of geodesics and the KPZ formula. We also discuss some of the main techniques for proving statements about the LQG metric, give examples of their use, and discuss some open problems. 
\end{abstract}

\tableofcontents

\section{Introduction}
\label{sec-intro}

\emph{Liouville quantum gravity} (LQG) is a family of models of random ``surfaces", or equivalently random ``two-dimensional Riemannian manifolds" which are in some sense canonical. The reason for the quotations is that, as we will see, LQG surfaces are too rough to be Riemannian manifolds in the literal sense. Such surfaces were first studied in the physics literature in the 1980's~\cite{polyakov-qg1,david-conformal-gauge,dk-qg,kpz-scaling}. 
The purpose of this article is give an overview of the construction of the distance function associated with an LQG surface (Section~\ref{sec-metric}) as well as some of its properties (Section~\ref{sec-properties}) and the main tools used for studying it (Section~\ref{sec-proofs}). We also discuss some open problems in Section~\ref{sec-open-problems}. In the rest of this section, we will give some basic background on the theory of LQG and its motivations.
\bigskip

\noindent \textbf{Acknowledgments.} J.\ Ding was partially supported by NSF grants DMS-1757479 and DMS-1953848. J.\ Dub\'edat was partially supported by NSF grant DMS-1512853. E.G.\ was partially supported by a Clay research fellowship. 

\subsection{Definition of LQG}

One can define LQG surfaces with the topology of any orientable surface (disks, spheres, torii, etc.), and all have the same local geometry.
We will be primarily interested in the local geometry, so for simplicity we will focus on LQG surfaces with the topology of the whole plane.\footnote{See~\cite{wedges,dkrv-lqg-sphere,grv-higher-genus,drv-torus,remy-annulus} for constructions of canonical LQG surfaces with various topologies.} 
 
To define LQG, we first need to define the Gaussian free field. 
The whole-plane \emph{Gaussian free field} (GFF) is the centered Gaussian process $h$ with covariances\footnote{Our choice of covariance function corresponds to normalizing $h$ so that its average over the unit circle is zero; see, e.g.,~\cite[Section 2.1.1]{vargas-dozz-notes}.}
\eqbn
\op{Cov}(h(z) , h(w)) = G(z,w) := \log \frac{\max\{|z|,1\}  \max\{|w|,1\}}{|z-w|} ,\quad\forall z,w\in \BB C .
\eqen 
Since $\lim_{w\rta z} G(z,w) = \infty$, the GFF is not a function. However, it still makes sense as a generalized function (i.e., a distribution). That is, if $\phi : \BB C \rta\BB R$ is smooth and compactly supported, then one can define the $L^2$ inner product $(h,\phi) = \int_{\BB C} h(z) \phi(z) \,d^2 z$ as a random variable. These random variables have covariances
\eqbn
\op{Cov}\left( (h,\phi) , (h,\psi) \right) = \int_{\BB C  \times \BB C } \phi(z) \psi(w) G(z,w) \,d^2z\,d^2w .
\eqen
 The reader can consult~\cite{shef-gff,pw-gff-notes,bp-lqg-notes} for more background on the GFF.
We have included a simulation of the GFF in Figure~\ref{fig-planar-map}, left.

More generally, we say that a random generalized function $h$ on $\BB C$ is a \emph{GFF plus a nice function} if $h = \wt h + f$, where $\wt h$ is the whole-plane GFF and $f : \BB C\rta \BB R$ is a (possibly random and $\wt h$-dependent) function which is continuous except at finitely many points. 

Let $\gamma \in (0,2]$, which will be the parameter for our LQG surfaces. 
A \emph{$\gamma$-LQG surface} parametrized by $\BB C$ is the random two-dimensional Riemannian manifold with Riemannian metric tensor
\eqb \label{eqn-ddk}
  e^{\gamma h(z)} (dx^2+dy^2) , \quad \text{for} \quad z = x+iy
\eqe 
where $dx^2+dy^2$ denotes the Euclidean metric tensor and $h$ is the whole-plane GFF, or more generally a whole-plane GFF plus a nice function.  
 
\subsection{Area measure and conformal covariance}
\label{sec-measure}

The Riemannian metric tensor~\eqref{eqn-ddk} is not well-defined since $h$ is not defined pointwise, so $e^{\gamma h}$ does not make literal sense. However, it is possible to make sense of various objects associated with~\eqref{eqn-ddk} rigorously using regularization procedures. The idea is to consider a collection of continuous functions $\{h_\ep\}_{\ep > 0}$ which converge to $h$ in some sense as $\ep\rta 0$, define objects associated with the Riemannian metric tensor~\eqref{eqn-ddk} with $h_\ep$ in place of $h$, then take a limit as $\ep \rta 0$. In this paper, we will discuss two objects which can be constructed in this way: the LQG area measure (to be discussed just below) and the LQG metric (which is the main focus of the paper). Other examples include the LQG length measure on Schramm-Loewner evolution-type curves~\cite{shef-zipper,benoist-lqg-chaos}, Liouville Brownian motion~\cite{grv-lbm,berestycki-lbm}, and the correlation functions for the random ``fields" $e^{\alpha h}$ for $\alpha  \in \BB R$~\cite{krv-dozz}. 
 
For simplicity, let us restrict attention to the case when $h$ is a whole-plane GFF. 
A convenient choice of $\{h_\ep\}$ is the convolution of $h$ with the heat kernel. 
For $t  > 0$ and $z\in\BB C$, we define the heat kernel $p_t(z) := \frac{1}{2\pi t} e^{-|z|^2/2t}$ and we define
\eqb \label{eqn-gff-convolve}
h_\ep^*(z) := (h*p_{\ep^2/2})(z) = \int_{\BB C} h(w) p_{\ep^2/2} (z  - w) \, d^2w ,\quad \forall z\in \BB C  
\eqe
where the integral is interpreted in the sense of distributional pairing. 

The easiest non-trivial object associated with~\eqref{eqn-ddk} to construct rigorously is the LQG area measure, or volume form. This is a random measure $\mu_h$ on $\BB C$ which is defined as the a.s.\ limit, with respect to the vague topology,\footnote{In the case when $\gamma = 2$, there is a log correction in the scaling factor, see~\cite{shef-deriv-mart,shef-renormalization,powell-critical-lqg}.}
\eqb \label{eqn-lqg-measure}
\mu_h = \lim_{\ep \rta 0} \ep^{\gamma^2/2} e^{\gamma h_\ep^*} \,d^2 z ,
\eqe 
where $d^2 z$ denotes Lebesgue measure on $\BB C$.
The reason for the normalizing factor $\ep^{\gamma^2/2}$ is that $\BB E[e^{\gamma h_\ep^*(z)}] \approx \ep^{-\gamma^2/2}$. 
The existence of the limit in~\eqref{eqn-lqg-measure} is a special case of the theory of Gaussian multiplicative chaos (GMC)~\cite{kahane,rhodes-vargas-review}. There are a variety of different ways of approximating $\mu_h$ which are all known to converge to the same limit; see~\cite{shef-kpz,shamov-gmc} for some results in this direction. 

The measure $\mu_h$ is mutually singular with respect to Lebesgue measure. In fact, it is supported on a dense subset of $\BB C$ of Hausdorff dimension $2-\gamma^2/2$; see, e.g.,~\cite[Section 3.3]{shef-kpz}. However, it has no atoms and assigns positive mass to every open subset of $\BB C$. 

The LQG area measure also satisfies a conformal covariance property. Let $U , \wt U\subset \BB C$ be open and let $f : \wt U \rta U$ be a conformal (bijective, holomorphic) map. Let 
\eqb \label{eqn-Q}
\wt h = h\circ \phi  +Q\log |\phi'| , \quad \text{where} \quad Q = \frac{2}{\gamma} + \frac{\gamma}{2} .
\eqe
Then $\wt h$ is a random generalized function on $\wt U$ whose law is locally absolutely continuous with respect to the law of $h$, so $\mu_{\wt h}$ can be defined.  
It is shown in~\cite[Proposition 2.1]{shef-kpz} that a.s.\ 
\eqb \label{eqn-lqg-coord}
\mu_{\wt h}(X) = \mu_h(\phi(X)), \quad \text{$\forall$ Borel set $X\subset U$.} 
\eqe 
We can think of the pairs $(U, h|_U)$ and $(\wt U , \wt h)$ as representing two different parametrizations of the same LQG surface. The relation~\eqref{eqn-lqg-coord} implies that the LQG area measure is an intrinsic function of the surface, i.e., it does not depend on the choice of parametrization.

The main focus of this article is the \emph{LQG metric}, i.e., the Riemannian distance function associated with the Riemannian metric tensor~\eqref{eqn-ddk}. 
This metric can be constructed via a similar regularization procedure as the measure, but the proof of convergence is much more involved. See Section~\ref{sec-metric} for details.

\begin{figure}[ht!]
\begin{center}
\includegraphics[width=.3\textwidth]{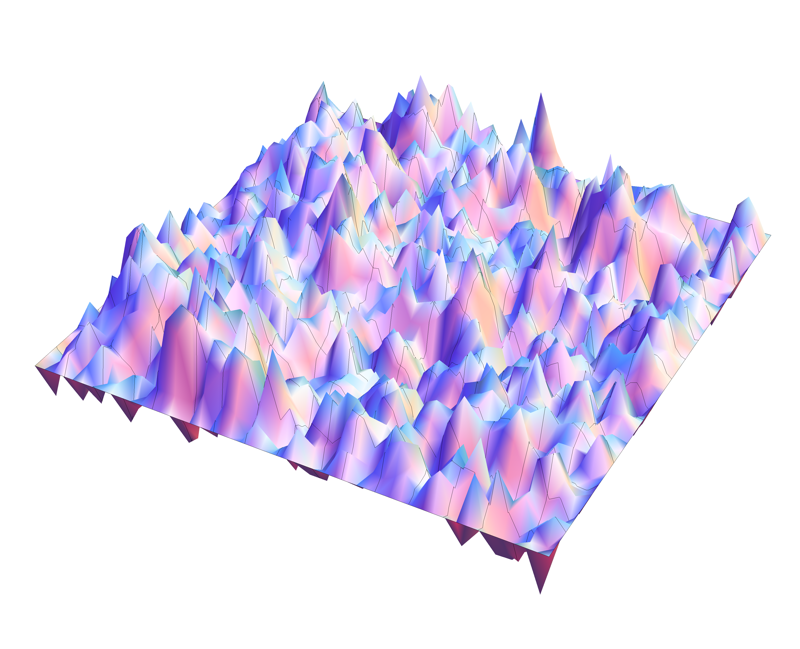}
\includegraphics[width=.6\textwidth]{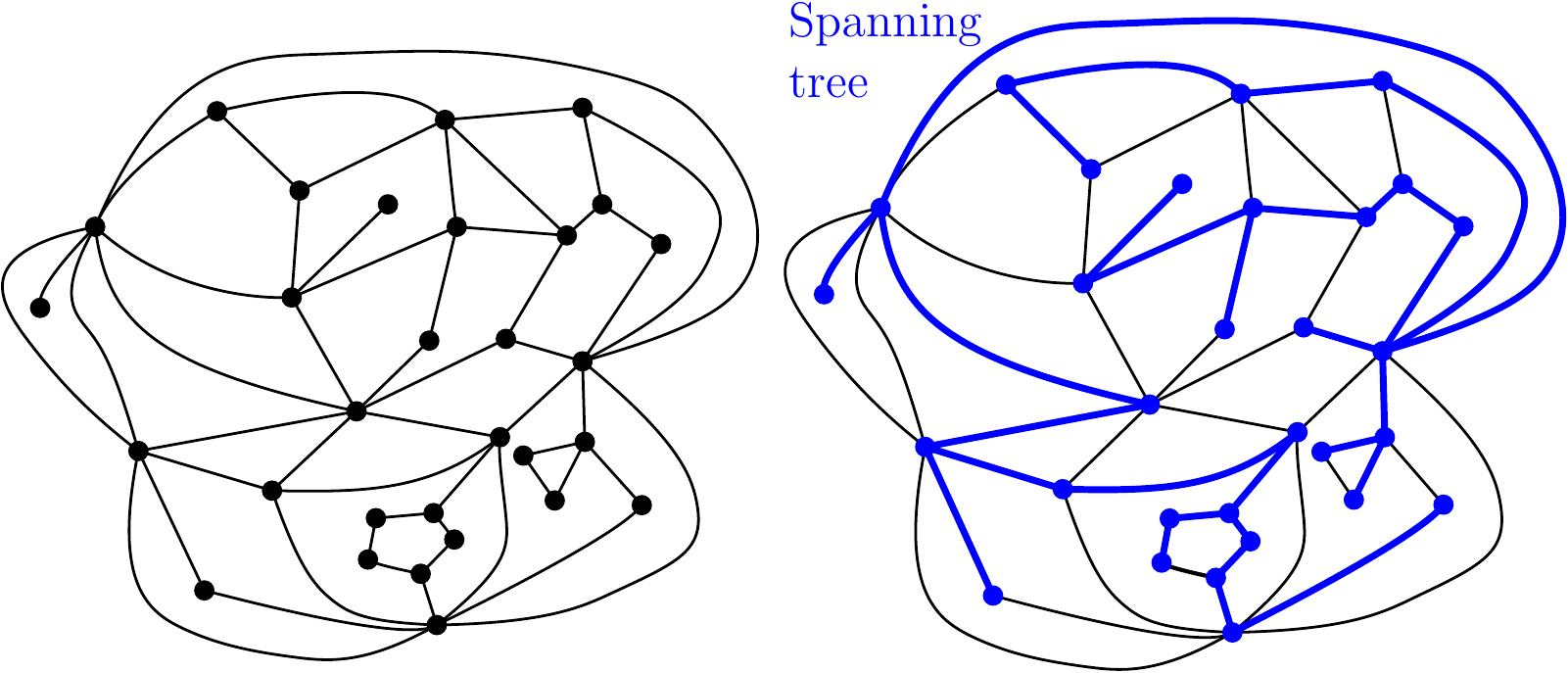} 
\caption{\label{fig-planar-map} \textbf{Left:} A simulation of the graph of a continuous function which approximates the GFF. \textbf{Middle.} A planar map. Equivalent representations of the same planar map can be obtained by applying an orientation-preserving homeomorphism from $\BB C$ to $\BB C$. \textbf{Right.} A spanning tree on the planar map. 
}
\end{center}
\end{figure}

\subsection{Motivation} \label{sec-motivation}
 
LQG was first studied by Polyakov~\cite{polyakov-qg1} in the 1980s in the context of string theory (we discuss Polyakov's motivation in Remark~\ref{remark-cc}). 
LQG is also of interest in conformal field theory since it is closely connected to Liouville conformal field theory, one of the simplest non-trivial conformal field theories. See~\cite{vargas-dozz-notes} for an overview of recent mathematical work on Liouville conformal field theory. 

One of the most important applications of LQG theory is the so-called Knizhnik-Polyakov-Zamolodchikov (KPZ) formula~\cite{kpz-scaling}, which gives a relationship between critical exponents for statistical mechanics models in random geometries and deterministic geometries.\footnote{The KPZ formula discussed here has no relation with Kardar-Parisi-Zhang equation from~\cite{kpz-fluctuation}, except that the initials of the authors for the two papers are the same.}
For example, this formula was used by Duplantier to give non-rigorous predictions for the Brownian intersection exponents~\cite{duplantier-bm-exponents} (the exponents were predicted earlier by Duplantier and Kwon~\cite{duplantier-kwon-brownian}). These predictions were later verified rigorously by Lawler, Schramm, and Werner in \cite{lsw-bm-exponents1,lsw-bm-exponents2,lsw-bm-exponents3} using SLE techniques. We discuss the KPZ formula in the context of the LQG metric in Section~\ref{sec-kpz}. 

Another reason to study LQG is that, at least conjecturally, it describes the large-scale behavior of discrete random geometries, such as random planar maps. A \emph{planar map} is a graph embedded in the plane so that no two edges cross, viewed modulo orientation-preserving homeomorphisms of the plane. See Figure~\ref{fig-planar-map}, middle, for an illustration.  
There are various interesting types of random planar maps, such as the following. 
\begin{itemize}
\item Uniform planar maps: consider the (finite) set of planar maps with a specified number $n\in\BB N$ of edges and choose an element of this set uniformly at random. 
\item Uniform planar maps with local constraints, such as triangulations (resp.\ quadrangulations), where each face has exactly 3 (resp.\ 4) edges. 
\item Decorated planar maps. Suppose, for example, that we want to sample a uniform pair $(M,T)$ consisting of a planar map $M$ with $n$ edges and a spanning tree $T$ on $M$ (i.e., a subgraph of $M$ which includes every vertex of $M$ and has no cycles). Under this probability measure, the marginal law of $M$ is not uniform; rather, the probability of seeing any particular planar map with $n$ edges is proportional to the number of spanning trees it admits. One can similarly consider planar maps decorated by statistical physics models (such as the Ising model or the FK model) or by various types of orientations on their edges.
\end{itemize}

It is believed that a large class of different types of planar maps converge to LQG in some sense. 
The parameter $\gamma$ depends on the type of planar map under consideration. Uniform planar maps, including maps with local constraints, correspond to $\gamma=\sqrt{8/3}$. This case is sometimes called ``pure gravity" in the physics literature. Other values of $\gamma$ correspond to planar maps decorated by statistical physics models. This case is sometimes called ``gravity coupled to matter". For example, the spanning tree-decorated maps discussed above are expected to converge to LQG with $\gamma=\sqrt 2$. 

For this article, the most relevant conjectured mode of convergence of random planar maps toward LQG is the following. View a planar map as a compact metric space, equipped with the graph distance. If we re-scale distances in this metric space appropriately, then as the number of edges tends to $\infty$ it should converge in the Gromov-Hausdorff sense to an LQG surface equipped with its LQG metric. So far, this type of convergence has only been proven for $\gamma=\sqrt{8/3}$, see Section~\ref{sec-lqg-tbm}. However, weaker connections between random planar maps and $\gamma$-LQG have been established rigorously for all $\gamma \in (0,2)$ using so-called \emph{mating of trees} theory. See~\cite{ghs-mating-survey} for a survey of this theory.

\section{Construction of the LQG metric}
\label{sec-metric}

\subsection{Liouville first passage percolation}
\label{sec-lfpp}

In analogy with the approximation scheme for the LQG measure in~\eqref{eqn-lqg-measure}, for a parameter $\xi > 0$, we define 
\eqb \label{eqn-lfpp}
D_{h}^\ep(z,w) := \inf_{P : z\rta w} \int_0^1 e^{\xi h_\ep^*(P(t))} |P'(t)| \,dt ,\quad \forall z,w\in\BB C ,\quad\forall \ep > 0 
\eqe
where the infimum is over all piecewise continuously differentiable paths $P : [0,1]\rta \BB C$ from $z$ to $w$. The metrics $D_h^\ep$ are sometimes referred to as $\ep$-\emph{Liouville first passage percolation} (LFPP). 

We want to choose the parameter $\xi$ in a manner depending on $\gamma$ so that the LFPP metrics~\eqref{eqn-lfpp} converge to the distance function associated with the metric tensor~\eqref{eqn-ddk}. To determine what $\xi$ should be, we use a heuristic scaling argument. From~\eqref{eqn-lqg-measure}, we see that scaling areas by $C  >0$ corresponds to replacing $h$ by $h + \frac{1}{\gamma} \log C$. On the other hand, from~\eqref{eqn-lfpp} we see that replacing $h$ by $h + \frac{1}{\gamma}\log C$ scales distances by a factor of $C^{\xi / \gamma}$. Hence $\xi/\gamma$ is the scaling exponent relating areas and distances. In other words, we want $\gamma / \xi  $ to be the ``dimension" of an LQG surface. 

It was shown in~\cite{dzz-heat-kernel,dg-lqg-dim} that there is an exponent $d_\gamma > 2$ which arises in various discrete approximations of LQG and which can be interpreted as the dimension of LQG. For example, $d_\gamma$ is the ball volume exponent for certain random planar maps~\cite[Theorem 1.6]{dg-lqg-dim}. Once the LQG metric has been constructed, one can show that $d_\gamma$ is its Hausdorff dimension~\cite{gp-kpz} (see Theorem~\ref{thm-dim}). The value of $d_\gamma$ is not known explicitly except that $d_{\sqrt{8/3}} = 4$. Computing $d_\gamma$ for general $\gamma \in (0,2]$ is one of the most important open problems in LQG theory. 

The above discussion suggests that one should take
\eqb \label{eqn-xi}
\xi = \frac{\gamma}{d_\gamma}. 
\eqe
It is shown in~\cite[Proposition 1.7]{dg-lqg-dim} that $\xi$ is an increasing function of $\gamma$, so for $\gamma \in (0,2]$, $\xi$ takes values in $(0,2/d_2]$. Estimates for $d_\gamma$~\cite{dg-lqg-dim,gp-lfpp-bounds} show that $2/d_2 \approx 0.41$. 

The definition of LFPP in~\eqref{eqn-lfpp} also makes sense for $\xi > 2/d_2$. In this regime, LFPP metrics do not correspond to $\gamma$-LQG with $\gamma\in(0,2]$. 
Rather, as we will explain in Section~\ref{sec-cc}, LFPP for $\xi  >2/d_2$ converges to a metric which is related to LQG with \emph{matter central charge} $\ccM \in (1,25)$, or equivalently $\gamma \in \BB C$ with $|\gamma|=2$. 

\begin{defn}
We refer to LFPP with $\xi < 2/d_2$, $\xi = 2/d_2$, and $\xi > 2/d_2$ as the \emph{subcritical}, \emph{critical}, and \emph{supercritical} phases, respectively.
\end{defn}

\begin{remark}  \label{remark-difficulty}
It is much more difficult to show the convergence of the approximating metrics~\eqref{eqn-lfpp} than it is to show the convergence of the approximating measures in~\eqref{eqn-lqg-measure}. One intuitive explanation for this is that the infimum in~\eqref{eqn-lfpp} introduces a substantial degree of non-linearity. The minimizing path in~\eqref{eqn-lfpp} depends on $\ep$, so one has to keep track of both the location of the minimizing path and its length, whereas for the measure one just has to keep track of the mass of a given set. One can think of the study of LFPP as the study of the extrema of the path-indexed random field whose value on each path is given by the integral in~\eqref{eqn-lfpp}.  
\end{remark}

\begin{remark} \label{remark-fpp}
The study of LFPP is very different from the study of ordinary first passage percolation (FPP), say on $\BB Z^2$. 
In ordinary FPP, the weights of the edges are i.i.d.\ and the law of the random environment is stationary with respect to spatial translations, neither of which are the case for LFPP (the law of the whole-plane GFF is only translation invariant modulo additive constant). 
However, for LFPP one has strong independence statements for the field at different Euclidean scales and one can get approximate spatial independence in certain contexts. See Sections~\ref{sec-annulus-iterate} and~\ref{sec-perc}. 
These independence properties are fundamental tools in the proof of the convergence of LFPP and the study of the limiting metric. 
\end{remark}

\begin{figure}[t!]
 \begin{center}
\includegraphics[width=0.4\textwidth]{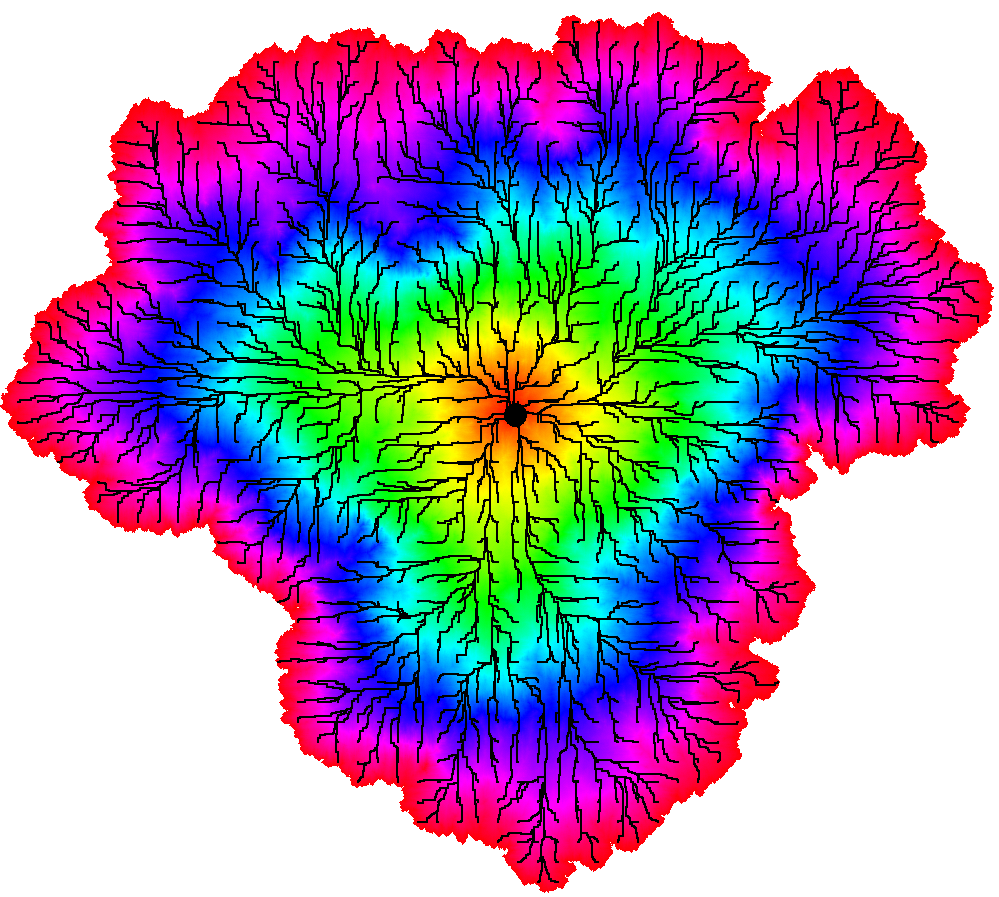} \hspace{0.05\textwidth}
\includegraphics[width=0.4\textwidth]{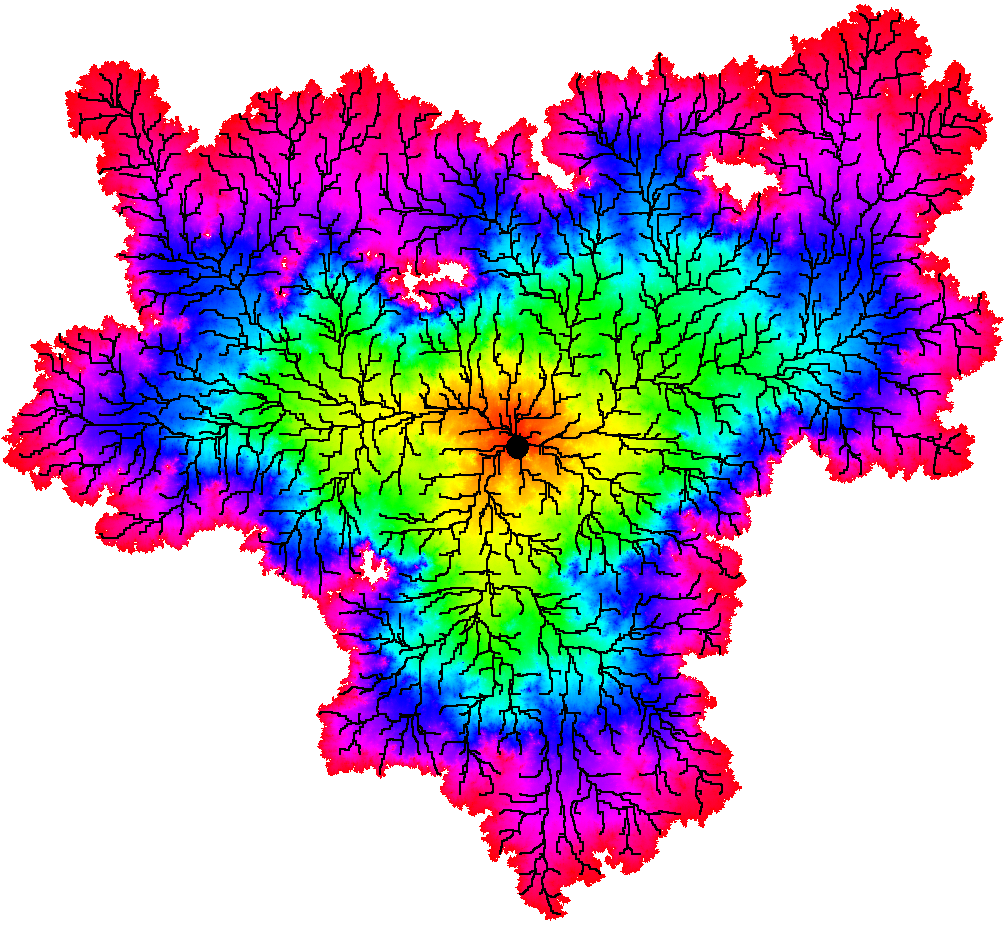} \\
\includegraphics[width=0.4\textwidth]{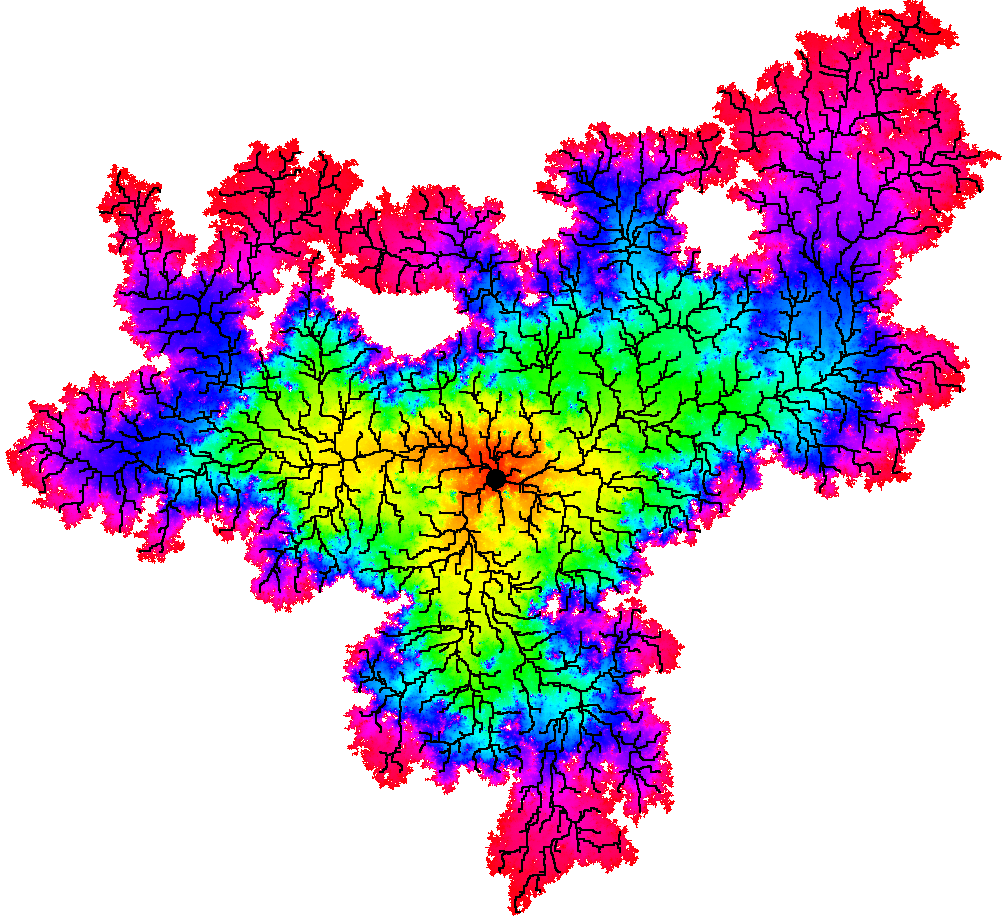} \hspace{0.05\textwidth}
\includegraphics[width=0.4\textwidth]{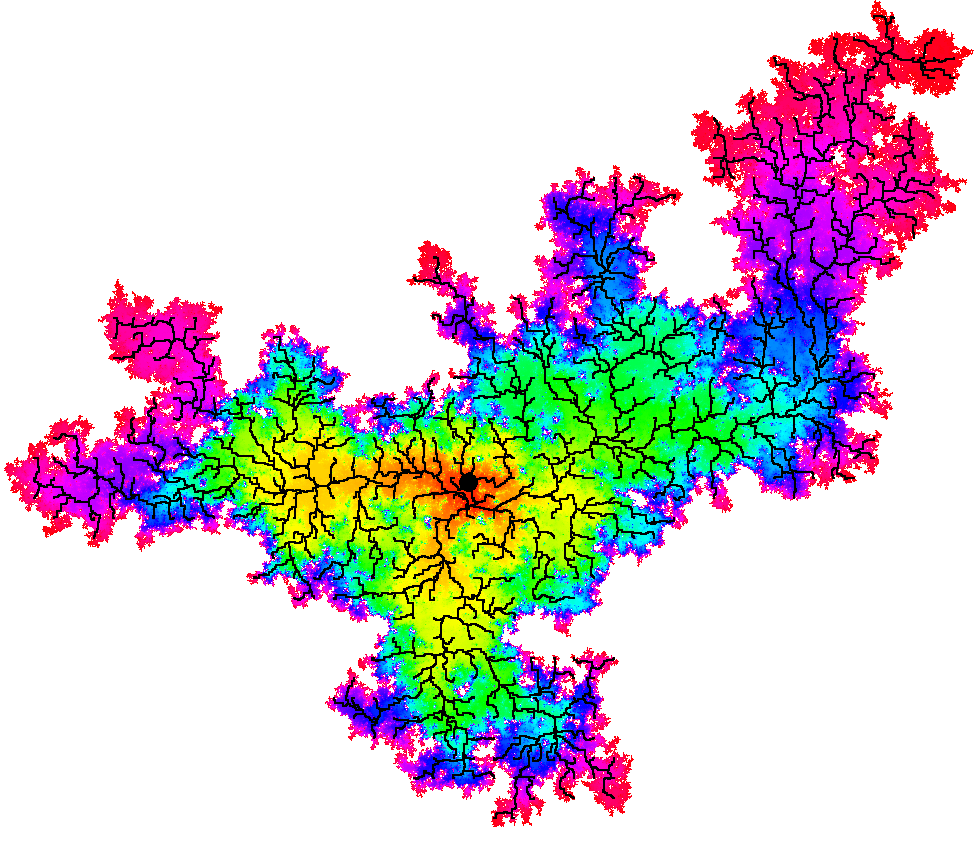}  
\vspace{-0.01\textheight}
\caption{Simulation of LFPP metric balls for $\xi = 0.2$ (top left), $\xi =  0.4$ (top right), $\xi = 0.6$ (bottom left), and $\xi = 0.8$ (bottom right). 
The values $\xi = 0.2,0.4$ are subcritical and correspond to $\gamma \approx 0.46 $ and $\gamma\approx 1.48$, respectively. The values $\xi = 0.6, 0.8$ are supercritical. 
The colors indicate distance to the center point (marked with a black dot) and the black curves are geodesics from the center point to other points in the ball. 
These geodesics have a tree-like structure, which is consistent with the confluence of geodesics results discussed in Section~\ref{sec-geodesic}. 
The pictures are slightly misleading in that the balls depicted do not have enough ``holes". In actuality, LQG metric balls have infinitely many complementary connected components for all $\xi > 0$, and have empty Euclidean interior for $\xi > 2/d_2$ (Section~\ref{sec-ball}). 
The simulation was produced using LFPP w.r.t.\ a discrete GFF on a $1024 \times 1024$ subset of $\BB Z^2$.  It is believed that this variant of LFPP falls into the same universality class as the variant in~\eqref{eqn-gff-convolve}. The geodesics go from the center of the metric ball to points in the intersection of the metric ball with the grid $20\BB Z^2$. The code for the simulation was provided by J.\ Miller.
}\label{fig-ball-sim}
\end{center}
\vspace{-1em}
\end{figure} 

\subsection{Convergence in the subcritical case}
\label{sec-subcritical}

\subsubsection{Tightness}
\label{sec-tight}

To extract a non-trivial limit of the metrics $D_h^\ep$, we need to re-normalize. We (somewhat arbitrarily) define our normalizing factor by
\eqb \label{eqn-gff-constant}
\frk a_\ep := \text{median of} \: \inf\left\{ \int_0^1 e^{\xi h_\ep^*(P(t))} |P'(t)| \,dt  : \text{$P$ is a left-right crossing of $[0,1]^2$} \right\} ,
\eqe  
where a left-right crossing of $[0,1]^2$ is a piecewise continuously differentiable path $P : [0,1]\rta [0,1]^2$ joining the left and right boundaries of $[0,1]^2$.
 
The value of $\frk a_\ep$ is not known explicitly (in contrast to the case of the LQG measure), but it is shown in~\cite[Proposition 1.1]{dg-supercritical-lfpp} that for each $\xi > 0$, there exists $Q = Q(\xi) > 0$ such that 
\eqb  \label{eqn-Q-def}
\frk a_\ep = \ep^{1 - \xi Q  + o_\ep(1)} ,\quad \text{as} \quad \ep \rta 0 . 
\eqe 
The existence of $Q$ is proven via a subadditivity argument, so the exact relationship between $Q$ and $\xi$ is not known. However, it is known that $Q \in (0,\infty)$ for all $\xi  > 0$, $Q$ is a continuous, non-increasing function of $\xi$, $\lim_{\xi\rta 0} Q(\xi) = \infty$, and $\lim_{\xi\rta\infty} Q(\xi) = 0$~\cite{dg-supercritical-lfpp,lfpp-pos}. See also~\cite{gp-lfpp-bounds,ang-discrete-lfpp} for bounds for $Q$ in terms of $\xi$. 

In the subcritical and critical cases, one has $\xi = \gamma/d_\gamma$ for some $\gamma \in (0,2]$ and 
\eqb \label{eqn-Q-subcrit}
Q(\gamma/d_\gamma) = \frac{2}{\gamma} + \frac{\gamma}{2} .
\eqe 
In other words, the value of $Q$ for LFPP is the same as the value of $Q$ appearing in the LQG coordinate change formula~\eqref{eqn-Q}. Furthermore, from~\eqref{eqn-Q-subcrit} we see that determining the relationship between $Q$ and $\xi$ in the subcritical case is equivalent to computing $d_\gamma$. 

The first major step in the construction of the LQG metric is to show that the re-scaled metrics $\frk a_\ep^{-1} D_h^\ep$ are tight, i.e., they admit subsequential limits in distribution. The first paper to prove a version of this was~\cite{ding-dunlap-lqg-fpp}, which showed that the metrics $\frk a_\ep^{-1} D_h^\ep$ are tight when $\xi$ is smaller than some non-explicit constant. The proof of this result was simplified in~\cite{df-lqg-metric}: most importantly,~\cite{df-lqg-metric} gave a simpler proof of the necessary RSW estimate (for all $\xi>0$) using a conformal invariance argument. Finally, tightness for the full subcritical regime $\xi\in (0,2/d_2)$ was proven in~\cite{dddf-lfpp}.

\begin{thm}[\!\!\cite{dddf-lfpp}] \label{thm-tight}
Assume that $\xi < 2/d_2$. The laws of the metrics $\{\frk a_\ep^{-1} D_h^\ep\}_{\ep > 0}$ are tight with respect to the topology of uniform convergence on compact subsets of $\BB C\times \BB C$. Every possible subsequential limit is a metric on $\BB C$ which induces the same topology as the Euclidean metric.
\end{thm}

Although the subsequential limit induces the same topology as the Euclidean metric, its geometric properties are very different. See Figure~\ref{fig-ball-sim} and Section~\ref{sec-properties}.

\subsubsection{Uniqueness}
\label{sec-unique}

The second major step is to show that the subsequential limit is unique. In fact, we want a stronger statement than just the uniqueness of the subsequential limit, since we would like to say that the limiting metric does not depend on the approximation procedure. To this end, the paper~\cite{gm-uniqueness} established an axiomatic characterization of the LQG metric. To state this characterization, we need some preliminary definitions. 

Let $\frk d$ be a metric on $\BB C$. For a path $P : [a,b]\rta \BB C$, we define its $\frk d$-length by
\eqb \label{eqn-length}
\op{len}(P;\frk d) :=  \sup_{T} \sum_{i=1}^{\# T} \frk d(P(t_i) , P(t_{i-1})) 
\eqe 
where the supremum is over all partitions $T : a= t_0 < \dots < t_{\# T} = b$ of $[a,b]$. We say that $\frk d$ is a \emph{length metric} if for each $z,w\in\BB C$, $\frk d(z,w)$ is equal to the infimum of the $\frk d$-lengths of all paths joining $z$ and $w$.

For an open set $U\subset \BB C$, we define the \emph{internal metric} of $\frk d$ on $U$ by  
\eqb \label{eqn-internal-metric}
\frk d(z,w;U) = \inf\left\{\op{len}(P;\frk d) : \text{$P$ is a path from $z$ to $w$ in $U$} \right\} ,\quad \forall z,w\in U . 
\eqe 
We note that $\frk d(z,w;U)$ can be strictly larger than the $\frk d(z,w)$ since all of the paths from $z$ to $w$ of near-minimal $\frk d$-length might exit $U$. 

The following is the axiomatic definition of the LQG metric from~\cite{gm-uniqueness}.  

\begin{defn}[LQG metric]
\label{def-metric}
Let $\mcl D'$ be the space of distributions (generalized functions) on $\BB C$, equipped with the usual weak topology.\footnote{
We do not care about how $D_h$ is defined on any subset of $\mcl D'$ which has measure zero for the law of any random distribution which is a GFF plus a continuous function. }   
For $\gamma\in (0,2)$, a \emph{$\gamma$-LQG metric} is a measurable functions $h\mapsto D_h$ from $\mcl D'$ to the space of metrics on $\BB C $ which induce the Euclidean topology with the following properties. Let $h$ be a GFF plus a continuous function on $\BB C$: i.e., $h = \wt h  +f$ where $\wt h$ is a whole-plane GFF and $f$ is a possibly random continuous function. Then the associated metric $D_h$ satisfies the following axioms. 
\begin{enumerate}[I.]
\item \textbf{Length space.} Almost surely, $D_h$ is a length metric. \label{item-metric-length} 
\item \textbf{Locality.} Let $U\subset\BB C$ be a deterministic open set. 
The $D_h$-internal metric $D_h(\cdot,\cdot ; U)$ is a.s.\ given by a measurable function of $h|_U$.  \label{item-metric-local}
\item \textbf{Weyl scaling.} Let $\xi$ be as in~\eqref{eqn-xi}. For a continuous function $f : \BB C \rta \BB R$, define
\eqb \label{eqn-metric-f}
(e^{\xi f} \cdot D_h) (z,w) := \inf_{P : z\rta w} \int_0^{\op{len}(P ; D_h)} e^{\xi f(P(t))} \,dt , \quad \forall z,w\in \BB C ,
\eqe 
where the infimum is over all $D_h$-continuous paths from $z$ to $w$ in $\BB C$ parametrized by $D_h$-length.
Then a.s.\ $ e^{\xi f} \cdot D_h = D_{h+f}$ for every continuous function $f: \BB C \rta \BB R$. \label{item-metric-f}
\item \textbf{Coordinate change for scaling and translation.} Let $r > 0$ and $z\in\BB C$. Almost surely, \label{item-metric-coord}
\eqbn
D_h(r u + z, rv + z) = D_{h(r\cdot+z)+Q\log r}(u,v) ,\quad\forall u,v\in\BB C ,\quad \text{where} \quad Q = \frac{2}{\gamma} + \frac{\gamma}{2} .
\eqen
\end{enumerate}
\end{defn}

The reason why we impose Axioms~\ref{item-metric-length} through~\ref{item-metric-f} is that we want $D_h$ to be the Riemannian distance function associated to the Riemannian metric tensor~\eqref{eqn-ddk}. Axiom~\ref{item-metric-coord} is analogous to the conformal coordinate change formula for the LQG area measure~\eqref{eqn-lqg-coord}, but restricted to translations and scalings. As in the case of the measure, it can be thought of as saying that the metric $D_h$ is intrinsic to the LQG surface, i.e., it does not depend on the choice of parametrization. The axioms in Definition~\ref{def-metric} imply a coordinate change formula for general conformal maps, including rotations; see~\cite[Remark 1.6]{gm-uniqueness} and~\cite{gm-coord-change}. 
 
The main result of~\cite{gm-uniqueness} is the following statement, whose proof builds on~\cite{dddf-lfpp,lqg-metric-estimates,local-metrics,gm-confluence}. 

\begin{thm}[\!\!\cite{gm-uniqueness}] \label{thm-unique}
For each $\gamma\in (0,2)$, there exists a $\gamma$-LQG metric. This metric is the limit of the re-scaled LFPP metrics $\frk a_\ep^{-1} D_h^\ep$ in probability w.r.t.\ the topology of uniform convergence on compact subsets of $\BB C\times\BB C$. Moreover, this metric is unique in the following sense: if $D_h$ and $\wt D_h$ are two $\gamma$-LQG metrics, then there is a deterministic constant $C>0$ such that a.s.\ $D_h(z,w) = C \wt D_h(z,w)$ for all $z,w\in\BB C$ whenever $h$ is a whole-plane GFF plus a continuous function.
\end{thm}

Due to Theorem~\ref{thm-unique}, we can refer to \emph{the} LQG metric, keeping in mind that this metric is only defined up to a deterministic positive multiplicative constant (the value of this constant is usually unimportant). 

Once Theorem~\ref{thm-unique} is established, it is typically easier to prove statements about the LQG metric directly from the axioms, as opposed to going back to the approximation procedure. We explain some of the techniques for doing so in Section~\ref{sec-proofs}.

\subsubsection{Weak LQG metrics}
\label{sec-weak}


The existence part of Theorem~\ref{thm-unique} of course follows from the tightness result in Theorem~\ref{thm-tight}, but not as directly as one might expect at first glance.  
It is relatively easy to check from the definition~\eqref{eqn-lfpp} that every possible subsequential limit of the re-scaled LFPP metrics $\frk a_\ep^{-1} D_h^\ep$ satisfies Axioms~\ref{item-metric-length}, \ref{item-metric-local}, and~\ref{item-metric-f} in Definition~\ref{def-metric}. See~\cite[Section 2]{lqg-metric-estimates} for details. 

Checking Axiom~\ref{item-metric-coord} is much more difficult. The reason is that re-scaling space changes the value of $\ep$ in~\eqref{eqn-lfpp}: for $\ep ,r>0$, one has~\cite[Lemma 2.6]{lqg-metric-estimates}
\eqbn
D_h^\ep(rz, rw) = r D_{h(r\cdot)}^{\ep/r}(z,w)  ,\quad \forall z,w\in \BB C.
\eqen
So, since we only have subsequential limits of $\frk a_\ep^{-1} D_h^\ep$, we cannot deduce that the subsequential limit satisfies an exact spatial scaling property. 

To get around this difficulty, we consider a weaker property than Axiom~\ref{item-metric-coord} which is sufficient for the proof of uniqueness. To motivate this property, let us consider how Axiom~\ref{item-metric-coord} is used in proofs about the LQG metric.
 
Assume that $h$ is a whole-plane GFF. For $z\in\BB C$ and $r>0$, let $h_r(z)$ be the average of $h$ over the circle $\bdy B_r(z)$ (see~\cite[Section 3.1]{shef-kpz} for the definition and basic properties of the circle average process). It is easy to see from the definition of the whole-plane GFF that for any $z\in\BB C$ and $r>0$,  
\eqb \label{eqn-gff-law}
h(r\cdot+z) - h_r(z) \eqD h .
\eqe 
Furthermore, from Weyl scaling and the LQG coordinate change formula (Axioms~\ref{item-metric-f} and~\ref{item-metric-coord}), a.s.\ 
\eqb \label{eqn-metric-scale}
D_{h(r\cdot+z) - h_r(z)}(u,v) = e^{-\xi h_r(z)} r^{-\xi Q} D_h(ru + z , rv + z) ,\quad\forall u ,v \in \BB C.
\eqe  
By~\eqref{eqn-gff-law} and~\eqref{eqn-metric-scale}  , 
\eqb \label{eqn-metric-law}
e^{-\xi h_r(z)} r^{-\xi Q} D_h(r\cdot + z , r \cdot + z)\eqD D_h .
\eqe
The relation~\eqref{eqn-metric-law} allows us to get estimates for $D_h$ which are uniform across different spatial locations and Euclidean scales. However, for many purposes one does not need an exact equality in law in~\eqref{eqn-metric-law}, but rather just an up-to-constants comparison. This motivates the following definition.

\begin{defn}[Weak LQG metric] \label{def-weak}
For $\gamma\in (0,2)$, a \emph{weak $\gamma$-LQG metric} is a measurable functions $h\mapsto D_h$ from $\mcl D'$ to the space of metrics on $\BB C $ which induce the Euclidean topology which satisfies Axioms~\ref{item-metric-length}, \ref{item-metric-local}, and~\ref{item-metric-f} in Definition~\ref{def-metric} plus the following further axioms.
\begin{enumerate}
\item[VI$'$.] \textbf{Translation invariance.} If $h$ is a whole-plane GFF, then for each fixed deterministic $z \in \BB C$, a.s.\ $D_{h(\cdot + z)} = D_h(\cdot + z , \cdot+z)$.  \label{item-metric-translate}
\item[V$'$.] \textbf{Tightness across scales.}  Suppose $h$ is a whole-plane GFF and for $z\in\BB C$ and $r>0$ let $h_r(z)$ be the average of $h$ over the circle $\bdy B_r(z)$. For each $r > 0$, there is a deterministic constant $\frk c_r > 0$ such that the set of laws of the metrics $\frk c_r^{-1} e^{-\xi h_r(0)} D_h (r \cdot , r\cdot)$ for $r > 0$ is tight (w.r.t.\ the local uniform topology). Furthermore, every subsequential limit of the laws of the metrics $\frk c_r^{-1} e^{-\xi h_r(0)} D_h (r \cdot  , r \cdot )$ is supported on metrics which induce the Euclidean topology on $\BB C$. 
\end{enumerate}
\end{defn}

From~\eqref{eqn-metric-law}, we see that every strong LQG metric is a weak LQG metric with $\frk c_r = r^{\xi Q}$. 
Furthermore, it is straightforward to check that every subsequential limit of LFPP is a weak LQG metric~\cite{lqg-metric-estimates}. 
In particular, Theorem~\ref{thm-tight} implies that there exists a weak LQG metric for each $\gamma\in(0,2)$. 
We note that most literature requires rather weak a priori bounds for the scaling constants $\frk c_r$ in Definition~\ref{def-weak}, but the recent paper~\cite{dg-polylog} shows that these bounds are unnecessary. 

It turns out that most statements which can be proven for LQG metrics can also be proven for weak LQG metrics. 
Using this,~\cite{gm-uniqueness} established the following statement.

\begin{thm}[Uniqueness of weak LQG metrics] \label{thm-weak-uniqueness}
Let $\gamma \in (0,2)$ and let $D_h$ and $\wt D_h$ be two weak $\gamma$-LQG metrics which have the \emph{same} values of $\frk c_r$ in Definition~\ref{def-weak}.
There is a deterministic constant $C > 0$ such that if $h$ is a whole-plane GFF plus a continuous function, then a.s.\ $D_h = C \wt D_h$. 
\end{thm} 

Let us now explain why Theorem~\ref{thm-weak-uniqueness} implies Theorem~\ref{thm-unique} (see~\cite[Section 1.4]{gm-uniqueness} for more details). 
If $D_h$ is a weak LQG metric and $b > 0$, then one can check that $D_{h(b\cdot )+Q\log b}(  \cdot / b ,   \cdot / b)$ is a weak LQG metric with the same scaling constants $\frk c_r$ as $D_h$. From this, one gets that $D_{h(b\cdot )+Q\log b}(  \cdot / b ,   \cdot / b) $ is a deterministic constant multiple of $D_h$. One can check that the constant has to be 1. This shows that $D_h$ satisfies Axiom~\ref{item-metric-coord} in Definition~\ref{def-metric}, i.e., $D_h$ is a strong LQG metric. In particular, $D_h$ is a weak LQG metric with scaling constants $r^{\xi Q}$. This holds for any possible weak LQG metric, so we infer that every weak LQG metric is a strong LQG metric and the weak LQG metric is unique up to constant multiples.

\begin{remark} \label{remark-lgd}
There are a few other ways to approximate the LQG metric besides LFPP, which are expected but not proven to give the same object. 
One possible approximation, called \emph{Liouville graph distance}, is based on the LQG area measure $\mu_h$: for $\ep > 0$ and $z,w\in\BB C$, we let $\wh D_h^\ep(z,w)$ be the minimal number of Euclidean balls of $\mu_h$-mass $\ep$ whose union contains a path from $z$ to $w$. 
The tightness of the metrics $\{\wh D_h^\ep\}_{\ep > 0}$, appropriately re-scaled, is proven in~\cite{ding-dunlap-lgd}, but the subsequential limit has not yet been shown to be unique. 

Another type of approximation is based on \emph{Liouville Brownian motion}, the ``LQG time" parametrization of Brownian motion on an LQG surface~\cite{grv-lbm,berestycki-lbm}. 
Roughly speaking, the idea here is that Liouville Brownian motion conditioned to travel a macroscopic distance in a small time should roughly follow an LQG geodesic. 
No one has yet established the tightness of any Liouville Brownian motion-based approximation scheme. However, the paper~\cite{dzz-heat-kernel} shows that the exponent for the Liouville heat kernel can be expressed in terms of the LQG dimension $d_\gamma$, which gives some rigorous connection between Liouville Brownian motion and the LQG metric.
\end{remark}

\subsection{The supercritical and critical cases}
\label{sec-supercritical}

\subsubsection{Convergence}
\label{sec-supercritical-tight}

Recall that LFPP is related to $\gamma$-LQG for $\gamma\in (0,2)$ in the subcritical case, i.e., when $\xi = \gamma/d_\gamma < 2/d_2 \approx 0.41\dots$. 
In this subsection we will explain what happens in the supercritical and critical cases, i.e., when $\xi \geq 2/d_2$. 

The tightness of critical and supercritical LFPP was established in~\cite{dg-supercritical-lfpp}. Subsequently, it was shown in~\cite{dg-uniqueness}, building on~\cite{pfeffer-supercritical-lqg}, that the subsequential limit is uniquely characterized by a list of axioms analogous to the ones in Definition~\ref{def-metric} (see~\cite[Section 1.3]{dg-uniqueness} for a precise statement). Unlike in the subcritical case, in the supercritical case the limiting metric $D_h$ is not a continuous function on $\BB C\times \BB C$, so one cannot work with the uniform topology. However, this metric is lower semicontinuous, i.e., for any $(z,w) \in \BB C \times \BB C$ one has
\eqb \label{eqn-lower-semicont}
D_h(z,w) \leq \liminf_{(z',w') \rta (z,w)} D_h(z',w') .
\eqe
In~\cite[Section 1.2]{dg-supercritical-lfpp} the authors describe a metrizable topology on the space of lower semicontinuous functions $\BB C\times \BB C \rta \BB R \cup \{\pm\infty\}$, based on the construction of Beer~\cite{beer-usc}. With this topology in hand, we can state the following generalization of Theorems~\ref{thm-tight} and~\ref{thm-unique}.

\begin{thm}[\!\!\cite{dg-supercritical-lfpp,pfeffer-supercritical-lqg,dg-uniqueness}] \label{thm-supercritical}
Let $\xi > 0$. The re-scaled LFPP metrics metrics $\{\frk a_\ep^{-1} D_h^\ep\}_{\ep > 0}$ converge in probability with respect to the topology on lower semicontinuous functions on $\BB C\times \BB C$. The limit $D_h$ is a metric on $\BB C$, except that it is allowed to take on infinite values. Moreover, $D_h$ is uniquely characterized (up to multiplication by a deterministic positive constant) by a list of axioms similar to the ones in Definition~\ref{def-metric}.  
\end{thm}

Let us be more precise about what we mean by allowing the metric to take on infinite values. For $\xi > 2/d_2$, it is shown in~\cite{dg-supercritical-lfpp} that if $D_h$ is as in Theorem~\ref{thm-supercritical}, then a.s.\ there is an uncountable dense set of \emph{singular points} $z\in\BB C$ such that 
\eqb \label{eqn-singular}
D_h(z,w) = \infty, \quad \forall w \in \BB C\setminus \{z\} .
\eqe
However, a.s.\ each fixed $z\in\BB C$ is not a singular point (so the singular points have Lebesgue measure zero) and any two non-singular points lie at finite $D_h$-distance from each other. Roughly speaking, if $\{h_r(z) : z\in\BB C , r > 0\}$ denotes the circle average process of $h$, then singular points correspond to points in $\BB C$ for which $\limsup_{r\rta 0} h_r(z) / \log r > Q$, where $Q$ is as in~\eqref{eqn-Q-def}~\cite[Proposition 1.11]{pfeffer-supercritical-lqg}. 

Due to the existence of singular points, for $\xi > 2/d_2$, the metric $D_h$ is not continuous with respect to the Euclidean metric on $\BB C\times \BB C$, but one can still show that the Euclidean metric is continuous with respect to $D_h$, see~\cite[Theorem 1.3]{dg-supercritical-lfpp} or~\cite[Proposition 1.10]{pfeffer-supercritical-lqg}.

In the critical case $\xi  =2/d_2$, which corresponds to $\gamma=2$, it is shown in~\cite{dg-critical-lqg} that $D_h$ induces the Euclidean topology on $\BB C$. In particular, there are no singular points for $\xi = 2/d_2$. We expect that the LFPP metrics $\frk a_\ep^{-1} D_h^\ep$ converge uniformly to $D_h$ in this case (not just with respect to the topology on lower semicontinuous functions), but this has not been proven.

\subsubsection{Central charge} 
\label{sec-cc}

\begin{figure}[ht!]
\begin{center}
\includegraphics[width=\textwidth]{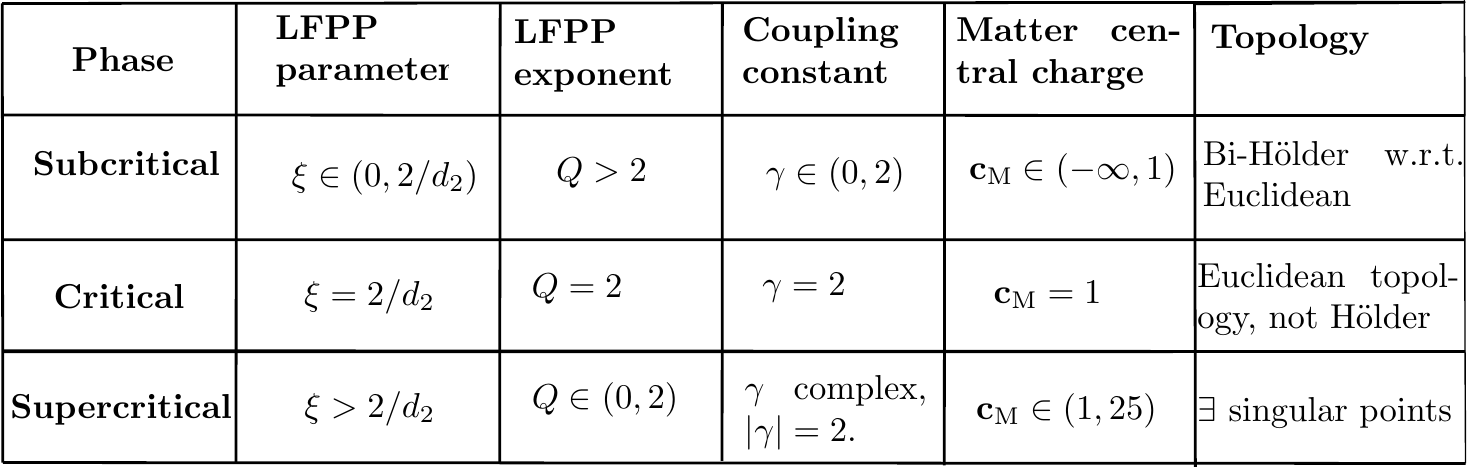}
\caption{\label{fig-c-phases-table} Table summarizing the phases for the LQG metric. 
}
\end{center}
\end{figure}

For $\gamma \in (0,2]$, the \emph{matter central charge} associated with $\gamma$-LQG is 
\eqb \label{eqn-cc}
\ccM = 25 - 6 Q^2 = 25 - 6\left(\frac{2}{\gamma} + \frac{\gamma}{2}\right)^2 \in (-\infty,1]. 
\eqe
Note that $\gamma=\sqrt{8/3}$ corresponds to $\ccM  =0$ and the critical case $\gamma=2$ corresponds to $\ccM  =1$.  
From physics heuristics, one expects that it should also be possible to define LQG, at least in some sense, in the case when the matter central charge is in $(1,25)$. However, this regime is much less well understood than the case when $\ccM \in (-\infty,1]$, even at a physics level of rigor. A major reason for this is that the formula~\eqref{eqn-cc} shows that $\ccM \in (1,25)$ corresponds to $\gamma \in \BB C$ with $|\gamma|=2$, so various formulas for LQG yield non-physical complex answers when $\ccM \in (1,25)$. See~\cite{ghpr-central-charge,apps-central-charge} for further discussion, references, and open problems concerning LQG with $\ccM \in (1,25)$. 

In light of~\eqref{eqn-Q-subcrit} and~\eqref{eqn-cc}, it is natural to define the matter central charge associated with LFPP for $\xi > 2/d_2$ by
\eqb \label{eqn-cc-supercrit}
\ccM = 25 - 6 Q(\xi)^2 ,
\eqe
where $Q(\xi)$ is the LFPP distance exponent as in~\eqref{eqn-Q-def}. One has $Q(\xi) \in (0,2)$ for $\xi > 2/d_2$, so~\eqref{eqn-cc-supercrit} gives $\ccM \in (1,25)$ for $\xi > 2/d_2$. 
Hence, the limit of supercritical LFPP can be interpreted as a metric associated with LQG with $\ccM \in (1,25)$. 
Since $\xi \mapsto Q(\xi)$ is continuous and non-increasing and $\lim_{\xi\rta\infty} Q(\xi ) = 0$~\cite[Proposition 1.1]{dg-supercritical-lfpp}, there is a $\xi > 2/d_2$ corresponding to each $\ccM \in (1,25)$. 

See Figure~\ref{fig-c-phases-table} for an table summarizing the phases for the LQG metric.

\begin{remark} \label{remark-cc}
From a physics perspective, an LQG surface with matter central charge $\ccM$ represents ``two dimensional gravity coupled to a matter field with central charge $\ccM$". 
Equivalently, an LQG surface parametrized by a domain $U$ should be a ``uniform sample from the space of Riemannian metric tensors $g$ on $U$, weighted by $(\det \Delta_g)^{-\ccM/2}$, where $\Delta_g$ is the Laplace Beltrami operator". 
This interpretation is far from being rigorous (e.g., since there is no uniform measure on the space of Riemannian metric tensors), but some partial progress using on regularization procedures has been made in~\cite{apps-central-charge}.    

The central charge also comes up in Polyakov's original motivation for LQG from string theory. If $\ccM$ is an integer, then roughly speaking an evolving string in $\BB R^{\ccM  -1}$ traces out a two-dimensional surface embedded in space-time $\BB R^{\ccM-1} \times \BB R$, called a \emph{world sheet}. Polyakov wanted to develop a theory of integrals over all possible surfaces embedded in $\BB R^{\ccM}$ as a string-theoretic generalization of the Feynman path integral (which is an integral over all possible paths). To do this one needs to define a probability measure on surfaces. It turns out that the ``right" measure on surfaces for this purpose is LQG with matter central charge $\ccM$. However, the most relevant case for string theory is $\ccM = 25$, which is outside the range of parameter values for which LQG can be defined probabilistically. 
\end{remark}

\subsection{Alternative construction and planar map connection for $\gamma=\sqrt{8/3}$} 
\label{sec-lqg-tbm}

In the special case when $\gamma=\sqrt{8/3}$, there is an earlier construction of the $\sqrt{8/3}$-LQG metric due to Miller and Sheffield~\cite{lqg-tbm1,lqg-tbm2,lqg-tbm3}. We will comment briefly on the main idea of this construction. See Miller's ICM paper~\cite{miller-icm} for a more detailed overview. 

The idea of the Miller-Sheffield construction is to first construct a candidate for LQG metric balls, then show that these balls are in fact the metric balls for a unique metric on $\BB C$. The candidates for LQG metric balls are generated using a random growth process called \emph{quantum Loewner evolution} (QLE), which is produced by ``re-shuffling" an SLE$_6$ curve in a random manner depending on $h$. The construction of this growth process and the proof that one can generate a metric from it both rely crucially on special symmetries for $\sqrt{8/3}$-LQG which are established in~\cite{wedges,sphere-constructions}, so the construction does not work for any other value of $\gamma$. 

The Miller-Sheffield metric satisfies the conditions of Definition~\ref{def-metric}, so Theorem~\ref{thm-unique} implies that it agrees with the $\sqrt{8/3}$-LQG metric constructed using LFPP. 
On the other hand, the construction using QLE gives a number of properties of the $\sqrt{8/3}$-LQG metric which are not apparent from the LFPP construction, for example various Markov properties for LQG metric balls and the fact that $d_{\sqrt{8/3}} = 4$. These properties can be proven directly using QLE, or can alternatively be deduced from analogous properties of the Brownian map together with the equivalence between the Brownian map and $\sqrt{8/3}$-LQG discussed just below. 

The papers~\cite{lqg-tbm1,lqg-tbm2} also establish a link between the $\sqrt{8/3}$-LQG metric and uniform random planar maps. This link comes by combining two big results:
\begin{itemize}
\item Le Gall~\cite{legall-uniqueness} and Miermont~\cite{miermont-brownian-map} showed independently that certain types of uniform random planar maps (namely, uniform $k$-angulations for $k=3$ or $k$ even), equipped with their graph distance, converge in the Gromov-Hausdorff sense to a random metric space called the \emph{Brownian map}. See~\cite{legall-sphere-survey,legall-brownian-geometry} for a survey of this work. 
\item Miller and Sheffield showed that there is a certain special variant of the GFF on $\BB C$ (corresponding to the so-called \emph{quantum sphere}) such that the sphere $\BB C\cup\{\infty\}$, equipped with the $\sqrt{8/3}$-LQG metric, is isometric to the Brownian map. This is done using the axiomatic characterization of the Brownian map from~\cite{tbm-characterization}. 
\end{itemize}

\begin{remark}
Building on the aforementioned work (and many additional papers), Holden and Sun~\cite{hs-cardy-embedding} showed the re-scaled graph distance on uniform triangulations embedded into the plane via the so-called \emph{Cardy embedding} converges to the $\sqrt{8/3}$-LQG metric with respect to a version of the uniform topology. This gives a stronger form of convergence than Gromov-Hausdorff convergence. 
 \end{remark}

\section{Properties of the LQG metric}
\label{sec-properties}

In this subsection we will discuss several properties of the LQG metric which have been established in the literature. Throughout, $h$ denotes a whole-plane GFF and $D_h$ denotes the associated LQG metric with a given parameter $\xi > 0$. We also let $Q$ be as in~\eqref{eqn-Q-def} and for $\xi \leq 2/d_2$ we let $\gamma\in (0,2)$ be such that $\xi = \gamma/d_\gamma$, so that $Q = 2/\gamma + \gamma/2$~\eqref{eqn-Q-subcrit}. 

\subsection{Dimension}
\label{sec-dim}

For $\Delta >0$, the $\Delta$-Hausdorff content of a compact metric space $(X,d)$ is
\eqbn
\inf\left\{\sum_{j=1}^\infty r_j^\Delta :\text{there is a covering of $X$ be $d$-metric balls with radii $\{r_j\}_{j\in\BB N}$}\right\}
\eqen
and the \emph{Hausdorff dimension} of $(X,d)$ is the infimum of the values of $\Delta$ for which the $\Delta$-Hausdorff content is zero.  
 
The following theorem follows from the combination of~\cite[Corollary 1.7]{gp-kpz} and~\cite[Proposition 1.14]{pfeffer-supercritical-lqg}.

\begin{thm} \label{thm-dim}
In the subcritical case, i.e., when $\gamma\in (0,2)$ and $\xi = \gamma/d_\gamma$, a.s.\ the Hausdorff dimension of $\BB C$, equipped with the $\gamma$-LQG metric, is equal to $d_\gamma$ (recall the discussion in Section~\ref{sec-lfpp}).
In the supercritical case, i.e., when $\xi  >2/d_2$, the Hausdorff dimension of $\BB C$, equipped with the LQG metric with parameter $\xi$, is $\infty$.
\end{thm}

As noted above, the value of $d_\gamma$ is not known except that $d_{\sqrt{8/3}}=4$, but upper and lower bounds for $d_\gamma$ have been proven in~\cite{dg-lqg-dim,gp-lfpp-bounds,ang-discrete-lfpp} (see Figure~\ref{fig-d-bound}). 
It is shown in~\cite[Theorem 1.2]{dg-lqg-dim} that $\gamma\mapsto d_\gamma$ is increasing and $\lim_{\gamma\rta 0} d_\gamma=2$. Hence, Theorem~\ref{thm-dim} implies that the LQG metric gets ``rougher" as $\gamma$ increases. We expect that the dimension of $\BB C$ with respect to the critical ($\gamma=2$) LQG metric is $d_2 = \lim_{\gamma\rta 2} d_\gamma\approx 4.8$, but this has not been proven.

It was shown in~\cite{afs-metric-ball} that for $\gamma\in (0,2)$, the Minkowski dimension of $(\BB C, D_h)$ is also equal to $d_\gamma$. 
We expect that in this case, the $d_\gamma$-Minkowski content measure for $D_h$ exists and is equal to the $\gamma$-LQG area measure $\mu_h$ from~\eqref{eqn-lqg-measure}. Similarly, the Hausdorff measure associated with $D_h$, for an appropriate gauge function, should exist and be equal to $\mu_h$. This has been proven for the Brownian map (which is equivalent to $\sqrt{8/3}$-LQG, recall Section~\ref{sec-lqg-tbm}) in~\cite{legall-measure}. 

\subsection{Quantitative estimates}
\label{sec-estimates}

The optimal H\"older exponents relating $D_h$ and the Euclidean metric can be computed in terms of $\xi$ and $Q$. For the subcritical (resp.\ supercritical) case, see~\cite[Theorem 1.7]{lqg-metric-estimates} (resp.~\cite[Proposition 1.10]{pfeffer-supercritical-lqg}).

\begin{prop}[H\"older continuity]
Let $U\subset\BB C$ be a bounded open set. Almost surely, for each $\delta>0$ there is a random $C>0$ such that
\eqbn
C^{-1} |z-w|^{\xi (Q+2) + \delta} \leq D_h(z,w) \leq  
\begin{cases}
C |z-w|^{\xi(Q-2) - \delta} ,\quad \xi < 2/d_2 \\
\infty , \quad \xi \geq 2/d_2 .
\end{cases}
\eqen
Furthermore, the exponents $\xi(Q+2)$ and $\xi(Q-2)$ are optimal.
\end{prop}

In the critical case when $\xi = 2/d_2$, equivalently $Q=2$, the metric $D_h$ is continuous with respect to the Euclidean metric but not H\"older continuous. Rather, the optimal upper bound for $D_h(z,w)$ is a power of $1/\log(|z-w|^{-1})$~\cite{dg-critical-lqg}. 

We also have moment bounds for point-to-point distances, set-to-set distances, and diameters. The following is a compilation of several results from~\cite{lqg-metric-estimates,pfeffer-supercritical-lqg}.

\begin{prop}[Moments]
For each distinct $z,w\in \BB C$, the distance $D_h(z,w)$ has a finite $p$th moment for all $p\in (-\infty,2Q/\xi)$. 
For any two disjoint compact connected sets $K_1,K_2\subset\BB C$ which are not singletons, $D_h(K_1,K_2)$ has finite moments of all positive and negative orders.
For $\xi < 2/d_2$, for any non-singleton compact set $K\subset \BB C$, the $D_h$-diameter $\sup_{z,w\in K} D_h(z,w)$ has a finite $p$th moment for all $p \in (-\infty,4 d_\gamma/ \gamma^2)$. 
\end{prop}

The moment bound for diameters is related to the fact that the LQG area measure has finite moments up to order $4/\gamma^2$ (see, e.g.,~\cite[Theorem 2.11]{rhodes-vargas-review}).

\subsection{Geodesics} 
\label{sec-geodesic}

Using basic metric space theory, one can show that a.s.\ for any two points $z,w\in\BB C$ with $D_h(z,w) < \infty$, there is a $D_h$-geodesic from $z$ to $w$, i.e., a path of minimal $D_h$-length  (see, e.g.,~\cite[Corollary 2.5.20]{bbi-metric-geometry} for the subcritical case and~\cite[Proposition 1.12]{pfeffer-supercritical-lqg} for the supercritical case). If $z$ and $w$ are fixed, then a.s.\ this geodesic is unique~\cite[Theorem 1.2]{mq-geodesics}. We give a short proof of this fact in Lemma~\ref{lem-geo-unique} below. 

It can be shown that the $D_h$-geodesics started from a specified point have a tree-like structure: two geodesics with the same starting point and different target points stay together for a non-trivial initial time interval. The property is called \emph{confluence of geodesics}, and can be seen in the simulations from Figure~\ref{fig-ball-sim}. 

We emphasize that confluence of geodesics is not true for a smooth Riemannian metric (such as the Euclidean metric). Rather, two geodesics for a smooth Riemannian metric with the same starting points and different target points typically intersect only at their starting point. 

Confluence of geodesics for the LQG metric was established in the subcritical case ($\xi < 2/d_2$) in~\cite{gm-confluence} and for general $\xi > 0$ in~\cite{dg-confluence}. Let us now state a precise version of this result, which is illustrated in Figure~\ref{fig-confluence} For $s > 0$ and $z\in\BB C$, let $\mcl B_s(z;D_h)$ be the $D_h$-metric ball of radius $s$ centered at $z$.

\begin{figure}[ht!]
\begin{center}
\includegraphics[width=.5\textwidth]{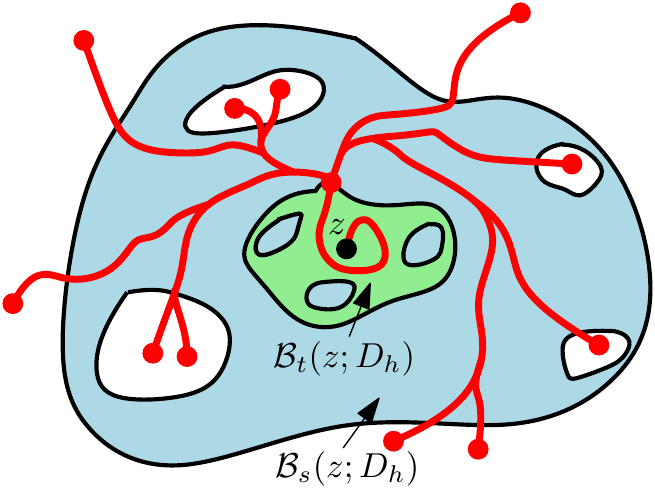}
\caption{\label{fig-confluence} Illustration of the statement of Theorem~\ref{thm-clsce}. The red curves are $D_h$-geodesics going from $z$ to points outside of the LQG metric ball $\mcl B_s(z;D_h)$. The theorem asserts that these geodesics all coincide until their first exit time from $\mcl B_t(z;D_h)$. 
}
\end{center}
\end{figure}

\begin{thm}[Confluence of geodesics] \label{thm-clsce}
Fix $z\in\BB C$. Almost surely, for each radius $s > 0$ there exists a radius $t \in (0,s)$ such that any two $D_h$-geodesics from $z$ to points outside of $\mcl B_s(z;D_h)$ coincide on the time interval $[0,t]$.
\end{thm} 

Theorem~\ref{thm-clsce} only holds a.s.\ for a fixed center point $z\in\BB C$. Almost surely, there is a Lebesgue measure zero set of points in $\BB C$ where Theorem~\ref{thm-clsce} fails. For example, if $P : [0,T]\rta\BB C$ is a $D_h$-geodesic, then the conclusion of Theorem~\ref{thm-clsce} fails for each $z\in P((0,T))$.

Confluence of geodesics is used in the proof of the uniqueness of the $\gamma$-LQG metric $\gamma\in (0,2)$ in~\cite{gm-uniqueness}. Roughly speaking, confluence is used to establish near-independence for events which depend on small neighborhoods of far-away points on a $D_h$-geodesic, despite the fact that $D_h$-geodesics are non-Markovian and do not depend locally on $h$. See~\cite{gm-uniqueness} for details. The proof of the uniqueness of the LQG metric for general $\xi > 0$ in~\cite{dg-uniqueness} does not use confluence of geodesics.  

\begin{remark}
Confluence of geodesics was previously established by Le Gall~\cite{legall-geodesics} for the Brownian map, which is equivalent to $\sqrt{8/3}$-LQG (see Section~\ref{sec-lqg-tbm}). This result was used in the proof of the uniqueness of the Brownian map in~\cite{legall-uniqueness,miermont-brownian-map}. 
Le Gall's proof was very different from the proof of Theorem~\ref{thm-clsce}.  
\end{remark}

Various extensions of the confluence property for $\gamma \in (0,2)$ are proven in~\cite{lqg-zero-one,gwynne-geodesic-network} and for $\gamma=\sqrt{8/3}$ in~\cite{akm-geodesics,mq-strong-confluence}. 

Little is known about the geometry of a single LQG geodesic. For example, we do not know the Hausdorff dimension of such a geodesic w.r.t.\ the Euclidean metric (the dimension w.r.t.\ the LQG metric is trivially equal to 1), and we do not have any exact description of its law. The strongest current results in this direction are an upper bound for the Euclidean dimension of an LQG geodesic~\cite[Corollary 1.10]{gp-kpz}, which is not expected to be optimal; and the fact LQG geodesics do not locally look like SLE$_\kappa$ curves for any value of $\kappa$~\cite{mq-geodesics}. We do not have a non-trivial lower bound for the Euclidean Hausdorff dimension of an LQG geodesic, but we expect that it is strictly greater than 1 (see~\cite{ding-zhang-geodesic-dim} for a closely related result for the geodesics for a version of LFPP). Finally, we mention the very recent work~\cite{bbg-lqg-geodesic}, which constructs a local limit of the GFF near a typical point of an LQG geodesic.

\subsection{Metric balls} 
\label{sec-ball}

From the simulations in Figure~\ref{fig-ball-sim}, one can see that LQG metric balls have a fractal-like geometry. Almost surely, the complement of each LQG metric ball has infinitely many connected components, in the subcritical, critical, and supercritical cases~\cite{lqg-zero-one,pfeffer-supercritical-lqg}. In fact, a.s.\ ``most" points on the boundary of the ball do not lie on any complementary connected component, but rather are accumulation points of arbitrarily small complementary connected components~\cite[Theorem 1.14]{lqg-zero-one}, \cite[Theorem 1.4]{dg-confluence}.  

In the subcritical and critical cases, i.e., when $\xi = \gamma/d_\gamma$ for $\gamma\in (0,2]$, the LQG metric induces the same topology as the Euclidean metric so a.s.\ each closed LQG metric ball is equal to the closure Euclidean interior. In contrast, in the supercritical case a.s.\ each LQG metric ball has empty Euclidean interior but positive Lebesgue measure. This is a consequence of the fact that the set of singular points from~\eqref{eqn-singular} is Euclidean-dense but has Lebesgue measure zero. 

In the subcritical case, it is shown in~\cite{gwynne-ball-bdy,lqg-zero-one} that a.s.\ the Hausdorff dimension of the boundary of a $\gamma$-LQG metric ball for $\gamma\in(0,2)$ w.r.t.\ the Euclidean (resp.\ LQG) metric is $2-\xi Q +\xi^2/2$ (resp.\ $d_\gamma-1$). 
We expect that these formulas are also valid for $\gamma =2$ (equivalently, $\xi = 2/d_2$). 

In the supercritical case $\xi > 2/d_2$, the LQG metric $D_h$ does not induce the Euclidean topology, so one has to make a distinction between the boundary with respect to the Euclidean topology or with respect to $D_h$. The boundary of a closed $D_h$-metric ball with respect to the Euclidean topology is equal to the ball itself (since the ball is Euclidean closed and has empty Euclidean interior), whereas the boundary with respect to $D_h$ is a proper subset of the ball~\cite[Section 1.2]{dg-confluence}. 
It is shown in~\cite[Proposition 1.14]{pfeffer-supercritical-lqg} that for $\xi > 2/d_2$, a.s.\ the Euclidean boundary of a $D_h$-metric ball (i.e., the whole $D_h$-metric ball) is not compact with respect to $D_h$ and has infinite Hausdorff dimension w.r.t.\ $D_h$.
We expect that the same is true for the $D_h$-boundary of a $D_h$-metric ball. 
The Hausdorff dimension of the Euclidean boundary of a $D_h$-metric ball with respect to the Euclidean metric is 2 since the metric ball has positive Lebesgue measure. 
The Hausdorff dimension of the $D_h$-boundary of a $D_h$-metric ball with respect to the Euclidean metric has not been computed rigorously. 

It is also of interest to consider the boundary of a single complementary connected component of an LQG metric ball. The Hausdorff dimension of such a boundary component w.r.t.\ the Euclidean or LQG metric is not known. However, it is known that, even in the supercritical case, each  boundary component is a Jordan curve and is compact and finite-dimensional w.r.t.\ $D_h$~\cite[Theorem 1.4]{dg-confluence}.

\subsection{KPZ formula} 
\label{sec-kpz}

The (geometric) Knizhnik-Polyakov-Zamolodchikov (KPZ) formula~\cite{kpz-scaling} is a formula which relates the ``Euclidean dimension" and the ``LQG dimension" of a deterministic set $X\subset \BB C$, or a random set independent from the GFF $h$. The first rigorous versions of the KPZ formula appeared in~\cite{shef-kpz,rhodes-vargas-log-kpz}. These papers defined the ``LQG dimension" in terms of the LQG area measure. There are several different versions of the KPZ formula in the literature which use different notions of dimension (see, e.g.,~\cite{aru-kpz,ghm-kpz,grv-kpz,bjrv-gmt-duality,benjamini-schramm-cascades}). Here, we state what is perhaps the most natural version of the KPZ formula, where we compare the Hausdorff dimensions of a set w.r.t.\ the LQG metric and the Euclidean metric. We start with the subcritical case, which is~\cite[Theorem 1.4]{gp-kpz}.

\begin{thm}[\!\!\cite{gp-kpz}] \label{thm-kpz} 
Let $\gamma \in (0,2)$ and recall that $\xi = \gamma/d_\gamma$ and $Q = 2/\gamma+\gamma/2$. 
Let $X\subset\BB C$ be a random Borel set which is independent from the GFF $h$ and let $\Delta_0$ be the Hausdorff dimension of $X$, equipped the Euclidean metric.
Also let $\Delta_h$ be the Hausdorff dimension of $X$, equipped with the $\gamma$-LQG metric $D_h$. 
Then a.s.\ 
\eqb  \label{eqn-kpz}
\Delta_h = \xi^{-1} (Q -  \sqrt{Q^2 - 2\Delta_0})   . 
\eqe 
\end{thm}

Theorem~\ref{thm-kpz} does not apply if $X$ is not independent from $h$. For example, the KPZ formula does not hold for the Hausdorff dimensions of LQG metric ball boundaries w.r.t.\ the Euclidean and LQG metrics, as discussed in Section~\ref{sec-ball}. However, one has inequalities relating the Hausdorff dimensions of an arbitrary set with respect to the Euclidean and LQG metrics, see~\cite[Theorem 1.8]{gp-kpz}. 

It is shown in~\cite[Theorem~1.15]{pfeffer-supercritical-lqg} that the KPZ formula of Theorem~\ref{thm-kpz} extends to the case when $\xi \geq 2/d_2$ (modulo some technicalities about the particular notion of ``fractal dimension" involved), with the following important caveat. When $\xi  > 2/d_2$, we have $Q  \in (0,2)$ and the right side of the formula~\eqref{eqn-kpz} is non-real when $\Delta_0 > Q^2/2$. The extension of the KPZ formula to the supercritical case coincides with~\eqref{eqn-kpz} when $\Delta_0 < Q^2/2$, and gives $\Delta_h = \infty$ when $\Delta_0 > Q^2/2$ (the case when $\Delta_0 =Q^2/2$ is not treated).

\section{Tools for studying the LQG metric}
\label{sec-proofs}

There are a few basic techniques which are the starting point of the majority of the proofs of statements involving the LQG metric.  In this subsection, we will discuss a few of the most important such techniques and provide some simple examples of their applications. Throughout, $h$ denotes a whole-plane GFF and $D_h$ denotes an LQG metric in the sense of Definition~\ref{def-metric}. For simplicity, we assume that we are in the subcritical case but our discussion applies in the critical and supercritical cases as well, with only minor modifications. 

\subsection{Adding a bump function} 
\label{sec-bump}

Suppose that $E$ is an event depending on the LQG metric $D_h$. For example, maybe we have two points $z,w\in\BB C$ and $E$ is the event that $D_h(z,w) > 100$, or that the $D_h$-geodesic from $z$ to $w$ stays in some specified open set. For many choices of $E$, it is straightforward to show that $\BB P[E] > 0$ via the following method. 
Let $\phi$ be a deterministic smooth, compactly supported function. It is easy to see from basic properties of the GFF that the laws of $h$ and $h+\phi$ are mutually absolutely continuous. See, e.g.,~\cite[Proposition 3.4]{ig1} for a proof. Using Weyl scaling (Axiom~\ref{item-metric-f}), we can choose $\phi$ so that with high probability, the event $E$ occurs with $h+\phi$ in place of $h$. The absolute continuity of the laws of $h+\phi$ and $h$ then implies that $\BB P[E] > 0$. Let us illustrate this idea by showing that an LQG geodesic stays in a specified open set with positive probability.

\begin{lem} \label{lem-geo-tube}
Let $z,w\in\BB C$ and let $U\subset\BB C$ be a connected open set which contains $z$ and $w$. 
With positive probability, every $D_h$-geodesic from $z$ to $w$ is contained in $U$.
\end{lem}
\begin{proof}
Let $V \subset V' \subset U$ be bounded, connected open sets containing $z$ and $w$ such that $\ol V \subset V'$ and $\ol V'\subset U$. It is a.s.\ the case that internal distance $D_h(z,w;V)$ is finite and the distance $D_h( V' , \bdy U)$ is positive, so we can find $C>0$ such that 
\eqb \label{eqn-geo-tube-reg}
\BB P\left[ D_h(z,w;V) \leq C ,\: D_h( V' , \bdy U) > C^{-1} \right] \geq \frac12 .  
\eqe 

Let $\phi$ be a smooth, non-negative bump function which is identically equal to $\frac{2}{\xi}\log C$ on $V$ and is identically equal to zero outside of $V'$. 
By Weyl scaling (Axiom~\ref{item-metric-f}) and since $\phi\equiv \frac{2}{\xi}\log C  $ on $V$, the $D_{h -  \phi}$-internal metric on $V$ is equal to $C^{-2}$ times the $D_h$-internal metric on $V$. Furthermore, since $\phi\equiv 0$ outside $V'$, we have $D_h( V' , \bdy U) = D_{h-\phi}(  V' , \bdy U)$. 
Therefore, if the event in~\eqref{eqn-geo-tube-reg} occurs, then
\eqbn
D_{h -  \phi}(z,w ; V) 
= C^{-2} D_h(z,w;V) 
\leq C^{-1} 
< D_h(\bdy V' , \bdy U) 
= D_{h-\phi}( V' , \bdy U) .
\eqen
In particular, $D_{h-\phi}(z,w)  < D_{h-\phi}(z,\bdy U)$. Therefore, no $D_{h-\phi}$-geodesic from $z$ to $w$ can exit $U$. This happens with probability at least $1/2$. Since the laws of $h-\phi$ and $h$ are mutually absolutely continuous, the lemma statement follows.
\end{proof}

In a similar vein, it is sometimes useful to add a \emph{random} bump function to $h$ in order to show that $D_h$ has certain ``typical" behavior with probability 1. To be more precise, again let $\phi$ be a smooth compactly supported bump function and let $X$ be a random variable which is uniform on $[0,1]$, sampled independently from $h$. Then the laws of $h$ and $h+X\phi$ are mutually absolutely continuous. So, if $E $ is an event depending on $D_h$, then to show that $\BB P[E] = 0$ it suffices to show that the probability that $E$ occurs with $h+X\phi$ in place of $h$ is zero. To show this latter statement, it suffices to show that a.s.\ the Lebesgue measure of the set of $x \in [0,1]$ such that $E$ occurs with $h+x\phi$ in place of $h$ is zero. Usually, it is possible to show that this set consists of at most a single point. Let us illustrate this technique by proving the uniqueness of $D_h$-geodesics between typical points. 

\begin{lem} \label{lem-geo-unique}
Fix distinct points $z,w\in\BB C$. Almost surely, there is a unique $D_h$-geodesic from $z$ to $w$. 
\end{lem}

Lemma~\ref{lem-geo-unique} was first established in~\cite[Theorem 1.2]{mq-geodesics} via an argument which is similar to, but more complicated than, the one we give here.
We emphasize that Lemma~\ref{lem-geo-unique} applies only for a fixed pair of points $z,w\in\BB C$. Almost surely, there are exceptional pairs of points which are joined by multiple $D_h$-geodesics. See~\cite{akm-geodesics,gwynne-geodesic-network,mq-strong-confluence} for a discussion of these exceptional pairs of points.

\begin{proof}[Proof of Lemma~\ref{lem-geo-unique}]
Let $U,V\subset\BB C$ be bounded open sets lying at positive distance from $z$ and $w$ such that $\ol V \subset U$. 
Let $E = E(U,V)$ be the event that the following is true: there are distinct $D_h$-geodesics $P,\wt P$ from $z$ to $w$ such that $P$ is disjoint from $U$ and $\wt P$ enters $V$. If there is more than one $D_h$-geodesic from $z$ to $w$, then $E(U,V)$ must occur for some choice of open sets $U, V$ which we can take to be finite unions of balls with rational centers and radii. Hence it suffices to fix $U$ and $V$ and show that $\BB P[E] = 0$. 

Let $\phi : \BB C \rta [0,1]$ be a smooth bump function which is identically equal to 1 on a neighborhood of $\ol V$ and which vanishes outside of $U$. 
For $x\in [0,1]$, let $E_x$ be the event that $E$ occurs with $h + x \phi$ in place of $h$. As explained above the lemma statement, it suffices to prove that a.s.\ the Lebesgue measure of the set of $x\in [0,1]$ for which $E_x$ occurs is 0. In fact, we will show that a.s.\ there is at most one values of $x\in [0,1]$ for which $E_x$ occurs. 

For this, it is enough to show that if $0 \leq x < y \leq 1$ and $E_x$ occurs, then $E_y$ does not occur. To see this, assume that $E_x$ occurs and let $P_x$ and $\wt P_x$ be the $D_{h-x\phi}$-geodesics as in the definition of $E_x$. By Weyl scaling (Axiom~\ref{item-metric-f}) and since $\phi$ is non-negative, we have $D_{h+y\phi}(u,v) \geq D_{h+x\phi}(u,v)$ for all $u,v\in\BB C$. 
Since $P_x$ does not enter $U$ and $\phi$ vanishes outside of $U$, we also have
\eqbn
D_{h+y\phi}(z,w)
\leq \op{len}\left( P_x ; D_{h+y\phi}\right)
= \op{len}\left( P_x ; D_{h+x\phi}\right)
= D_{h+x\phi}(z,w) ,
\eqen 
where here we recall the notation for length w.r.t.\ a metric from~\eqref{eqn-length}. 
Hence
\eqb \label{eqn-xy-dist}
D_{h+y\phi}(z,w) = D_{h+x\phi}(z,w).
\eqe 

Now suppose that $\wt P : [0,T]\rta \BB C$ is any path from $z$ to $w$ which enters $V$. We will show that $\wt P$ is not a $D_{h+y\phi}$-geodesic, which implies that $E_y$ does not occur. Indeed, there must be a positive-length interval of times $[a,b]$ such that $P([a,b]) \subset \phi^{-1}(1)$. We therefore have
\alb
\op{len}\left( \wt P ;  D_{h+y\phi} \right)
&= \op{len}\left( \wt P|_{[0,a] \cup [b,T]} ;  D_{h+y\phi} \right) + \op{len}\left( \wt P|_{[a,b]} ;  D_{h+y\phi} \right) \notag\\
&\geq \op{len}\left( \wt P|_{[0,a] \cup [b,T]} ;  D_{h+x\phi} \right) + e^{\xi (y-x)} \op{len}\left( \wt P|_{[a,b]} ;  D_{h+x\phi} \right)  \quad \text{(by Axiom~\ref{item-metric-f})} \notag\\
&\geq \op{len}\left( \wt P   ;  D_{h+x\phi} \right) + ( e^{\xi (y-x)}  -1 ) \op{len}\left( \wt P|_{[a,b]} ;  D_{h+x\phi} \right)  \notag\\
&> D_{h+x\phi}(z,w) \quad \text{(by Axiom~\ref{item-metric-length})}  \\
&= D_{h+y\phi}(z,w) \quad \text{(by~\eqref{eqn-xy-dist}) }.
\ale 
\end{proof}

\begin{remark}
If $\phi$ is a deterministic smooth bump function, then the proof of~\cite[Proposition 3.4]{ig1} shows that the Radon-Nikodym derivative of the law of $h+\phi$ w.r.t.\ the law of $h$ is given by
\eqbn
\exp\left((h,\phi)_\nabla - \frac12 (\phi,\phi)_\nabla \right) 
\eqen
where $(f,g)_\nabla := \int_{\BB C} \nabla f(z) \cdot \nabla g(z) \,d^2z$ is the Dirichlet inner product. One can use this explicit expression for the Radon-Nikodym derivative together with arguments of the sort discussed above to estimate the probabilities of certain rare events for the LQG metric. For example, this is the key idea in the computation of the dimension of a boundary of an LQG metric ball in~\cite{gwynne-ball-bdy}.
\end{remark}

\subsection{Independence across concentric annuli} 
\label{sec-annulus-iterate}

Another key tool in the study of the LQG metric is the fact that the restrictions of the GFF to disjoint concentric annuli (viewed modulo additive constant) are nearly independent. In particular, suppose that we have a sequence of events $\{E_{r_k}\}_{k\in\BB N}$ depending on the restrictions of $h$ to disjoint concentric annuli. If we have a lower bound for $\BB P[E_{r_k}]$ which is uniform in $k$, then for $K\in\BB N$ the number of $k\in \{1,\dots,K\}$ for which $E_{r_k}$ occurs can be compared to a binomial random variable. This leads to the following lemma, which is a special case of~\cite[Lemma~3.1]{local-metrics}.

\begin{lem}  \label{lem-annulus-iterate}
Fix $0 < s_1< s_2 < 1$. Let $z\in\BB C$ and let $\{r_k\}_{k\in\BB N}$ be a decreasing sequence of positive real numbers such that $r_{k+1} / r_k \leq s_1$ for each $k\in\BB N$. Let $\{E_{r_k} \}_{k\in\BB N}$ be events such that for each $k\in \BB N$, the event $E_{r_k} $ is a.s.\ determined by the restriction of $h - h_{r_k}(z)$ to the Euclidean annulus $B_{s_2 r_k}(z) \setminus B_{s_1 r_k}(z)$, where $h_{r_k}(z)$ denotes the circle average.  
\begin{enumerate}
\item \label{item-annulus-iterate} For each $a > 0$, there exists $p = p(a, s_1,s_2) \in (0,1)$ and $c = c(a, s_1,s_2) > 0$ such that if  
\eqb \label{eqn-annulus-iterate-prob}
\BB P\left[ E_{r_k}  \right] \geq p , \quad \forall k\in\BB N  ,
\eqe 
then 
\eqb \label{eqn-annulus-iterate}
\BB P\left[ \text{$\exists k\in \{1,\dots,K\}$ such that $E_{r_k}$ occurs} \right] \geq 1 - c e^{-a K} ,\quad\forall K \in \BB N. 
\eqe  
\item For each $p \in (0,1)$, there exists $a = a(p,s_1,s_2) > 0$ and $c = c(p, s_1,s_2) > 0$ such that if~\eqref{eqn-annulus-iterate-prob} holds, then~\eqref{eqn-annulus-iterate} holds. \label{item-annulus-iterate-pos}  
\end{enumerate}
\end{lem}

We emphasize that the numbers $p$ and $c$ in assertion~\ref{item-annulus-iterate} and the numbers $a$ and $c$ is assertion~\ref{item-annulus-iterate-pos} do \emph{not} depend on $z$ or on $\{r_k\}$ (except via $s_1,s_2$). 
The idea of Lemma~\ref{lem-annulus-iterate} was first used in~\cite{mq-geodesics}, and the general version stated here was first formulated in~\cite{local-metrics}. 
To illustrate the use of Lemma~\ref{lem-annulus-iterate}, we will explain a typical application: a polynomial upper bound for the probability that a $D_h$-geodesic gets near a point.

\begin{lem} \label{lem-geo-hit}
For each $\gamma \in (0,2)$, there exists $\alpha = \alpha(\gamma ) > 0$ and $c=c(\gamma) > 0$ such that the following is true. 
For each $z\in \BB C$ and each $\ep > 0$, the probability that there is a $D_h$-geodesic between two points in $\BB C\setminus B_{\ep^{1/2}}(z)$ which enters $B_\ep(z)$ is at most $c \ep^\alpha$.
\end{lem}

Roughly speaking, Lemma~\ref{lem-geo-hit} says that ``most" points in $\BB C$ are not hit by $D_h$-geodesics except at their endpoints. 
Lemma~\ref{lem-geo-hit} immediately implies that the Hausdorff dimension of every LQG geodesic w.r.t.\ the Euclidean metric is strictly less than 2. 
Similar (but more complicated) ideas to the ones in the proof of Lemma~\ref{lem-geo-hit} are used in the proof of confluence of geodesics in~\cite{gm-confluence,dg-confluence}.

Let us now proceed with the proof of Lemma~\ref{lem-geo-hit}.
The first step is to define the events for which we will apply Lemma~\ref{lem-annulus-iterate}.  
To lighten notation, we introduce the following terminology.

\begin{defn} \label{def-across-around}
For a Euclidean annulus $A\subset \BB C$, we define $D_h(\text{across $A$})$ to be the $D_h$-distance between the inner and outer boundaries of $A$. We define $D_h(\text{around $A$})$ to be the infimum of the $D_h$-lengths of paths in $A$ which separate the inner and outer boundaries of $A$.
\end{defn}

Both $D_h(\text{across $A$})$ and $D_h(\text{around $A$})$ are determined by the internal metric of $D_h$ on $A$, so by Axiom~\ref{item-metric-local} these quantities are a.s.\ determined by $h|_A$.  

For $z\in\BB C$ and $r> 0$, let
\eqb  \label{eqn-hit-event}
E_r(z) := \left\{ D_h\left(\text{around $B_{3r}(z) \setminus B_{2r}(z)$}\right)  <  D_h\left(\text{across $B_{2r}(z) \setminus B_r(z)$} \right) \right\} .
\eqe
As noted above, Axiom~\ref{item-metric-local} implies that $E_r(z)$ is a.s.\ determined by $h|_{B_{3r}(z) \setminus B_r(z)}$. In fact, adding a constant to $h$ results in scaling $D_h$-distances by a constant (Axiom~\ref{item-metric-f}), so adding a constant to $h$ does not affect whether $E_r(z)$ occurs. Hence $E_r(z)$ is a.s.\ determined by $(h - h_{4r}(z)) |_{B_{3r}(z) \setminus B_r(z)}$. 

\begin{lem} \label{lem-hit-prob}
There exists $\alpha = \alpha(\gamma) > 0$ and $c = c(\gamma ) > 0$ such that for each $z\in\BB C$ and each $\ep > 0$, 
\eqbn
\BB P\left[ \text{$\exists r \in \left[\ep , \frac14 \ep^{1/2}\right]$ such that $E_r(z)$ occurs}  \right] \geq 1 - c \ep^\alpha .
\eqen
\end{lem}
\begin{proof}
Using a ``subtracting a bump function" argument as discussed in Section~\ref{sec-bump}, one can show that $p := \BB P[E_1(0)] > 0$.
From~\eqref{eqn-metric-law}, we see $\BB P[E_r(z)]$ does not depend on $z$ or $r$. 
Hence $\BB P[E_r(z)] = p$ for each $z\in\BB C$ and $r>0$. We now apply Lemma~\ref{lem-annulus-iterate} with $r_k = 4^{-k} \ep^{1/2}$ and $K = \lfloor \frac12 \log_4 \ep^{-1} \rfloor$. Then $r_k  \in [\ep , \frac14 \ep^{1/2}]$ for each $k\in \{1,\dots,K\}$, so part~\ref{item-annulus-iterate-pos} of Lemma~\ref{lem-annulus-iterate} shows that there exists $a = a(\gamma) > 0$ and $c = c(\gamma) > 0$ such that 
\eqbn
\BB P\left[ \text{$\exists r \in [\ep , \ep^{1/2}]$ such that $E_r(z)$ occurs}  \right] \geq 1 - c p^{a K}  .
\eqen 
This last quantity is at least $1 - c \ep^\alpha$ for an appropriate $\alpha >0$ depending on $p,a$ (hence on $\gamma$).
\end{proof}

\begin{proof}[Proof of Lemma~\ref{lem-geo-hit}]
By Lemma~\ref{lem-hit-prob}, it suffices to show that if there is an $r \in [\ep , \frac14 \ep^{1/2}]$ such that $E_r(z)$ occurs, then no $D_h$-geodesic between two points in $\BB C\setminus B_{\ep^{1/2}}(z)$ can enter $B_\ep(z)$. 
Indeed, assume that $E_r(z)$ occurs, let $u,v\in\BB C\setminus B_{\ep^{1/2}}(z)$, and let $P$ be a path from $u$ to $v$ which hits $B_r(z) \supset B_\ep(z)$. We will show that $P$ is not a $D_h$-geodesic. By the definition~\eqref{eqn-hit-event} of $E_r(z)$, there is a path $\pi$ in $B_{3r}(z) \setminus B_{2r}(z)$ which disconnects the inner and outer boundaries of this annulus and has $D_h$-length strictly less than $D_h(\text{across $B_{2r}(z) \setminus B_r(z)$})$. Let $\sigma$ (resp.\ $\tau$) be the first (resp.\ last) time that $P$ hits $\pi$. Since $P$ hits $B_r(z)$ and $u,v \notin B_{3r}(z)$, the path $P$ crosses between the inner and outer boundaries of $B_{2r}(z) \setminus B_r(z)$ between times $\sigma$ and $\tau$. Hence 
\eqb
\left(\text{$D_h$-length of $P|_{[\sigma,\tau]}$}\right)  \geq D_h(\text{across $B_{2r}(z) \setminus B_r(z)$}) .
\eqe
But, since $P(\tau) , P(\sigma) \in \pi$, 
\allb
D_h(P(\sigma),P(\tau)) 
\leq \left(\text{$D_h$-length of $\pi$} \right)  
&< D_h(\text{across $B_{2r}(z) \setminus B_r(z)$}) \notag\\ 
&\leq \left(\text{$D_h$-length of $P|_{[\sigma,\tau]}$}\right) .
\alle
This implies that $P$ is not a $D_h$-geodesic since it is not the $D_h$-shortest path from $P(\sigma)$ to $P(\tau)$.
\end{proof}

\subsection{White noise decomposition}
\label{sec-perc}

A convenient way to approximate the GFF is by convolving the heat kernel with a space-time white noise. 
To explain this, let $W$ be a space-time white noise on $\BB C\times [0,\infty)$, i.e., $\{(W,f) : f\in L^2(\BB C\times [0,\infty))\}$ is a centered Gaussian process with covariances $\BB E[(W,f) (W,g) ]  = \int_\BB C\int_0^\infty f(z,s) g(z,s) \,ds \, dz$. For $f\in L^2(\BB C\times [0,\infty))$ and Borel measurable sets $A\subset\BB C$ and $I\subset [0,\infty)$, we slightly abuse notation by writing 
\eqbn
\int_A\int_I f(z,s) \, W(dz,ds) := (W , f \BB 1_{A\times I} ) .
\eqen

As in~\eqref{eqn-gff-convolve}, we denote the heat kernel by $p_t(z) := \frac{1}{2\pi t} e^{-|z|^2/2t}$.
Following~\cite[Section 3]{ding-goswami-watabiki}, we define the centered Gaussian process
\eqb \label{eqn-wn-decomp}
\wh h_t (z) := \sqrt\pi \int_{\BB C} \int_{t^2}^1 p_{s/2}(z-w) \, W(dw,ds)  ,\quad \forall t \in [0,1] , \quad \forall z\in \BB C .
\eqe 
We write $\wh h  := \wh h_0$. 
By~\cite[Lemma 3.1]{ding-goswami-watabiki} and Kolmogorov's criterion, each $\wh h_t$ for $t \in (0,1]$ admits a continuous modification.  
The process $\wh h$ does not admit a continuous modification, but makes sense as a distribution: indeed, it is easily checked that its integral against any smooth compactly supported test function is Gaussian with finite variance.
  
The process $\wh h$ is in some ways more convenient to work with than the GFF thanks to the following symmetries, which are immediate from the definition. 
\begin{itemize}
\item \textit{Rotation/translation/reflection invariance.} The law of $\{\wh h_t : t\in [0,1]  \}$ is invariant with respect to rotation, translation, and reflection of the plane.
\item \textit{Scale invariance.} For $\delta \in (0,1]$, one has $\{(\wh h_{\delta t } - \wh h_\delta)(\delta \cdot)  : t \in [0,1]  \} \eqD \{\wh h_t : t\in [0,1]\}$. 
\item \textit{Independent increments.} If $0 \leq t_1\leq t_2 \leq t_3 \leq t_4 \leq 1$, then $\wh h_{t_2} - \wh h_{t_1}$ and $\wh h_{t_4} - \wh h_{t_3}$ are independent. 
\end{itemize}
 
One property which $\wh h$ does not possess is spatial independence. To get around this, it is sometimes useful to work with a truncated variant of $\wh h $ where we only integrate over a ball of finite radius. To this end, we let $\phi : \BB C\rta[0,1]$ be a smooth bump function which is equal to 1 on the ball $B_{1/20}(0)$ and which vanishes outside of $B_{1/10}(0)$. For $t\in [0,1]$, we define 
\eqb \label{eqn-wn-truncate}
\wh h_t^\tr(z) := \sqrt\pi \int_{t^2}^1 \int_{\BB C} p_{s/2}(z-w) \phi(z-w) \, W(dw,dt) .
\eqe 
We also set $\wh h^\tr := \wh h^\tr_0$. 
As in the case of $\wh h$, it is easily seen from the Kolmogorov continuity criterion that each $\wh h^\tr_t$ for $t\in (0,1]$ a.s.\ admits a continuous modification.
The process $\wh h^\tr$ does not admit a continuous modification and is instead viewed as a random distribution.  

The key property enjoyed by $\wh h^\tr$ is spatial independence: if $A,B\subset \BB C$ with $\op{dist}(A,B) \geq 1/5$, then $\{\wh h^\tr_t|_A : t\in [0,1]\}$ and $\{\wh h^\tr_t|_B : t\in [0,1]\}$ are independent. Indeed, this is because $\{\wh h^\tr_t|_A : t\in [0,1]\}$ and $\{\wh h^\tr_t|_B : t\in [0,1]\}$ are determined by the restrictions of the white noise $W$ to the disjoint sets $B_{1/10}(A) \times \BB R_+$ and $B_{1/10}(B)\times \BB R_+$, respectively.  
Unlike $\wh h$, the distribution $\wh h^\tr$ does not possess any sort of scale invariance but its law is still invariant with respect to rotations, translations, and reflections of $\BB C$.   

The following lemma, which is proven in the same manner as~\cite[Lemma 3.1]{dg-lqg-dim}, tells us that the distributions $\wh h$ and $\wh h^\tr$ and the whole-plane GFF can all be compared up to constant-order additive errors. 

\begin{lem} \label{lem-gff-compare} 
Suppose $U\subset \BB C$ is a bounded open set. There is a coupling $(h ,  \wh h , \wh h^\tr)$ of a whole-plane GFF normalized so that $h_1(0) = 0$ and the fields from~\eqref{eqn-wn-decomp} and~\eqref{eqn-wn-truncate} such that the following is true. For any $h^1,h^2 \in \{h ,  \wh h , \wh h^\tr\}$, the distribution $(h^1-h^2)|_U$ a.s.\ admits a continuous modification and there are constants $c_0,c_1 > 0$ depending only on $U$ such that for $A>1$, 
\eqb \label{eqn-gff-compare}
\BB P\left[\max_{z\in U} |(h^1-h^2)(z)| \leq A \right] \geq 1 - c_0 e^{-c_1 A^2} .
\eqe 
\end{lem}
 
Lemma~\ref{lem-gff-compare} implies that each of $\wh h$ and $\wh h^\tr$ is a GFF plus a continuous function. Hence we can define the LQG metrics $D_{\wh h}$ and $D_{\wh h^\tr}$. The metric $D_{\wh h^\tr}$ is particularly convenient to work with due to the aforementioned finite range of dependence property of $\wh h^\tr$. 
This property allows one to use percolation-style arguments in order to produce large clusters of Euclidean squares where certain ``good" events occur. We refer to~\cite{ding-goswami-watabiki,dzz-heat-kernel,dg-lqg-dim,gms-poisson-voronoi} for examples of this sort of argument. 

The white noise decomposition also plays a key role in the proofs of tightness of LFPP in~\cite{ding-dunlap-lqg-fpp,df-lqg-metric,dddf-lfpp,dg-supercritical-lfpp}. In fact, these papers first prove tightness of LFPP defined using the white noise decomposition~\eqref{eqn-wn-decomp} in place of the functions $h_\ep^*$, then transfer to $h_\ep^*$ using a comparison lemma which is similar in spirit to Lemma~\ref{lem-gff-compare} (see~\cite[Section 6.1]{dddf-lfpp}).

\section{Open problems}
\label{sec-open-problems}

Here we highlight some of the most important open problems concerning the LQG metric. Much more substantial lists of open problems can be found in~\cite{gm-uniqueness,ghpr-central-charge}.

\begin{prob} \label{prob-d}
For $\gamma \in (0,2)$, compute the Hausdorff dimension $d_\gamma$ of $\BB C$, equipped with the $\gamma$-LQG metric. 
More generally, for $\xi > 0$ determine the relationship between the parameters $Q$ and $\xi$ of~\eqref{eqn-Q-def}. 
\end{prob}

Due to~\eqref{eqn-xi} and~\eqref{eqn-Q-subcrit}, computing $d_\gamma$ for $\gamma \in (0,2)$ is equivalent to finding the relationship between $Q$ and $\xi$ for $\xi \in (0,2/d_2)$. As noted above, the only known case is $d_{\sqrt{8/3}}=4$, equivalently $Q(1/\sqrt 6) = 5/\sqrt 6$. One indication of the difficulty of computing $Q$ in terms of $\xi$ is that the relationship between $Q$ and $\xi$ is not universal for LFPP defined using different log-correlated Gaussian fields~\cite{dzz-nonuniversality}. 

Many quantities associated with LQG surfaces and random planar maps can be expressed in terms of $d_\gamma$ (or $\xi$ and $Q$), such as the optimal H\"older exponents relating the LQG metric and the Euclidean metric~\cite{lqg-metric-estimates}, the Hausdorff dimension of the boundary of an LQG metric ball~\cite{gwynne-ball-bdy}, and the ball volume exponent for certain random planar maps~\cite{dg-lqg-dim}. Solving Problem~\ref{prob-d} would lead to exact formulas for these quantities. 
 
We do not have a guess for the formula relating $Q$ and $\xi$, nor do we know whether an explicit formula exists. The best-known prediction from the physics literature, due to Watabiki~\cite{watabiki-lqg}, is equivalent to $ Q = 1/\xi - \xi$ for $\xi \in (0,2/d_2)$. The prediction was proven to be false in~\cite{ding-goswami-watabiki}, at least for small values of $\xi$ (equivalently, small values of $\gamma$). An alternative proposal, put forward in~\cite{dg-lqg-dim}, is that $Q = 1/\xi - 1/\sqrt 6$ for $\xi \in (0,2/d_2)$. This formula has not been disproven for any value of $\xi \in (0,2/d_2)$, but it (like Watabiki's prediction) is inconsistent with the result of~\cite{lfpp-pos}, which shows that $Q > 0$ for all $\xi > 0$. We expect that both of the above predictions are false for all but finitely many values of $\xi$. 

The best known rigorous bounds relating $\xi$ and $Q$ are obtained in~\cite{dg-lqg-dim,gp-lfpp-bounds,ang-discrete-lfpp}. See Figure~\ref{fig-d-bound} for a graph of these bounds.

\begin{figure}[ht!]
\begin{center}
\includegraphics[width=0.4\textwidth]{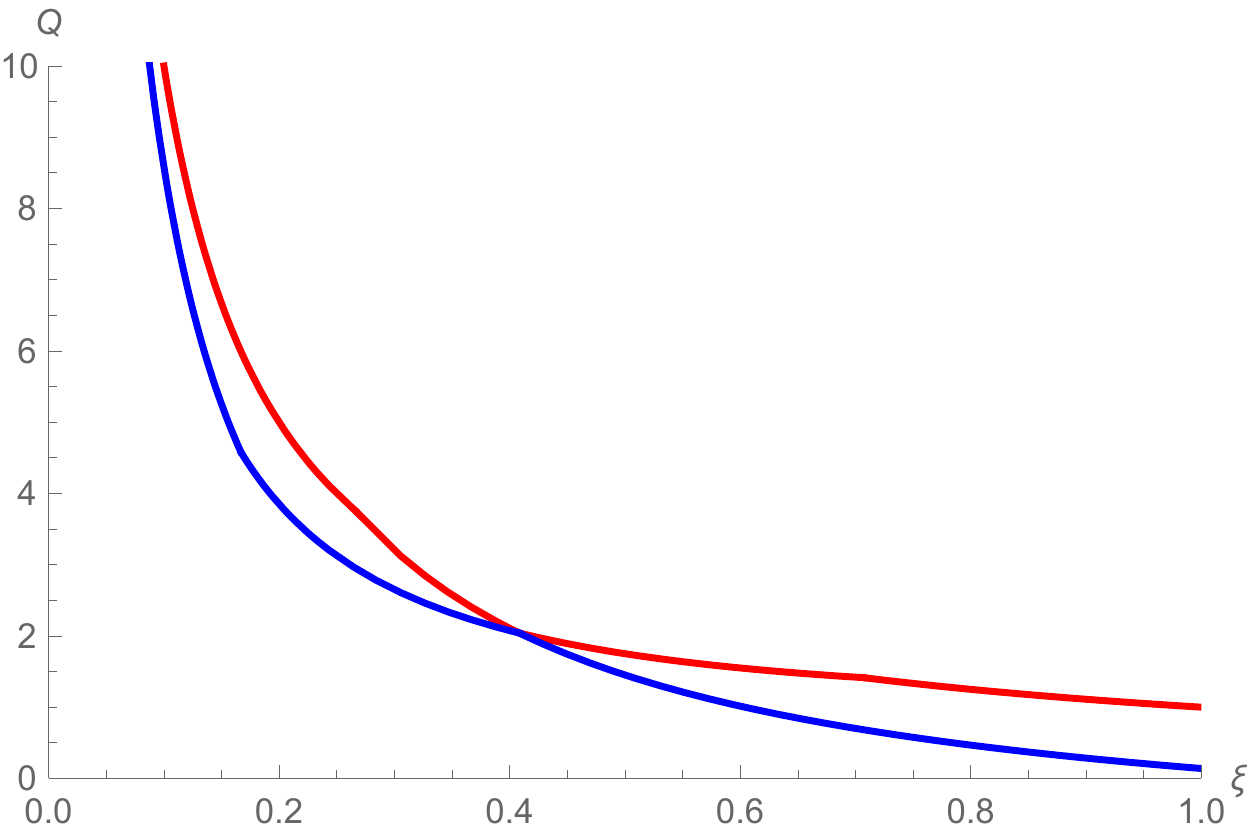} \hspace{10pt}
\includegraphics[width=0.4\textwidth]{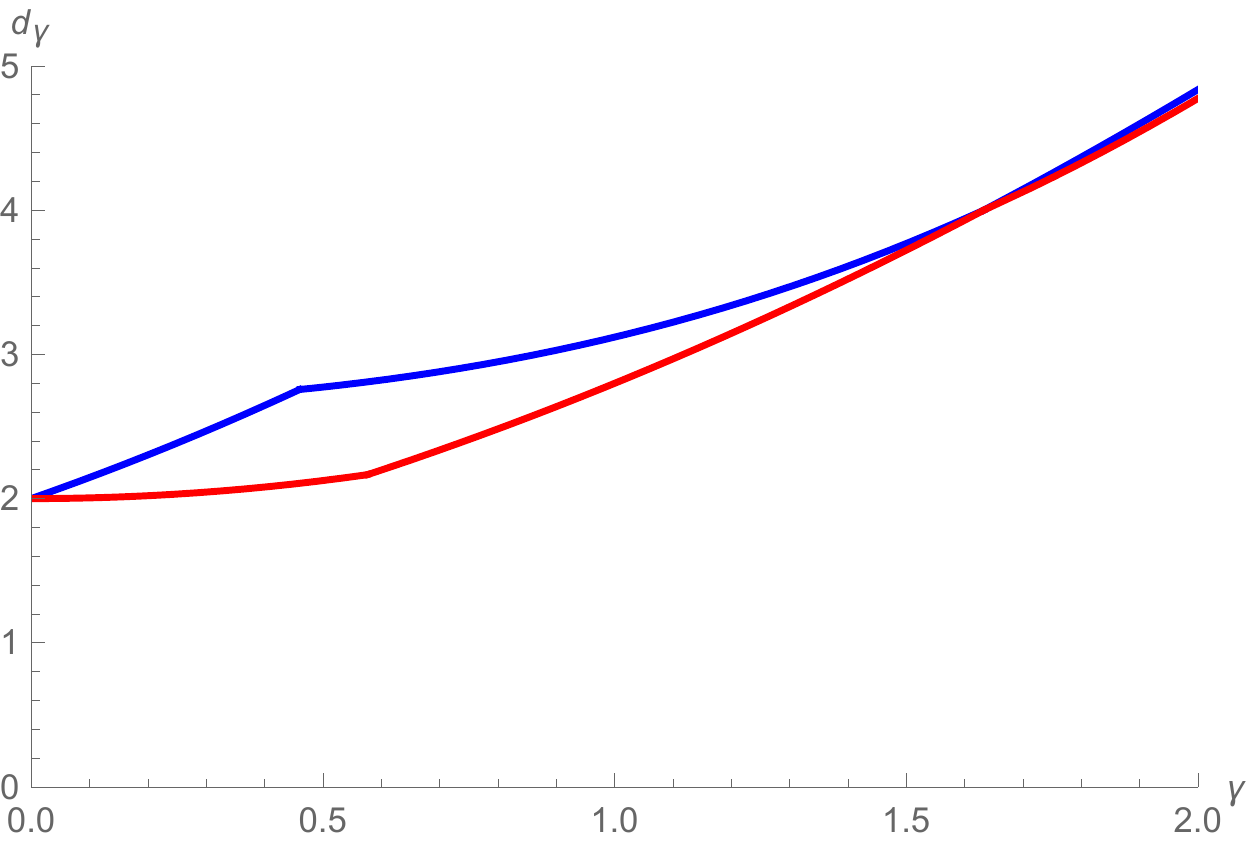} 
\caption{\label{fig-d-bound} \textbf{Left:} Plot of the best known upper (blue) and lower (red) bounds for $Q$ as a function of $\xi$. 
\textbf{Right:} Plot of the best-known bounds for $d_\gamma$ as a function of $\gamma$. 
}
\end{center}
\end{figure}

Our next open problem concerns the relationship between LQG surfaces and random planar maps. 

\begin{prob} \label{prob-map}
Show that for each $\gamma\in (0,2]$, appropriate types of random planar maps, equipped with their graph distance (appropriately rescaled), converge in the Gromov-Hausdorff sense to $\gamma$-LQG surfaces equipped with the $\gamma$-LQG metric.
\end{prob}

As discussed in Section~\ref{sec-motivation}, the value of $\gamma$ depends on the type of random planar map under consideration. For example, uniform random planar maps correspond to $\gamma=\sqrt{8/3}$, planar maps weighted by the number of spanning trees they admit correspond to $\gamma=\sqrt 2$, and planar maps weighted by the partition function of the critical Ising model on the map correspond to $\gamma=\sqrt 3$. So far, Problem~\ref{prob-map} has only been solved for $\gamma =\sqrt{8/3}$, see Section~\ref{sec-lqg-tbm}. 

Problem~\ref{prob-map} can be made more precise by specifying the scaling factor for the planar maps as well as the particular types of LQG surfaces one should get in the limit.
For concreteness, for $n\in\BB N$ consider the case of a random planar map $M_n$ with the topology of the sphere, having $n$ total edges. Then $M_n$, equipped with its graph distance re-scaled by $n^{-1/d_\gamma}$, should converge in the Gromov-Hausdorff sense to the quantum sphere, a special type of LQG surface which is defined in~\cite{wedges,dkrv-lqg-sphere} (the definitions are proven to be equivalent in~\cite{ahs-sphere}). Similar statements apply for random planar maps with other topologies, such as the disk, plane, or half-plane. 

Finally, we mention a third open problem which has not appeared elsewhere. For $\alpha \in\BB R$, let $\mcl T_h^\alpha$ be the set of $\alpha$-thick points of $h$, i.e., the points $z\in\BB C$ for which $\limsup_{\ep\rta 0} h_\ep(z)  / \log\ep^{-1} = \alpha$. Such points exist if and only if $\alpha \in [-2,2]$~\cite{hmp-thick-pts} For a set $X$, the function which takes $\alpha$ to the Hausdorff dimension of $X\cap \mcl T_h^\alpha$ (w.r.t.\ the LQG metric or the Euclidean metric) can be thought of as a sort of ``quantum multifractal spectrum" of $X$. 

\begin{prob} \label{prob-geodesic}
Let $\xi > 0$ and let $P  $ be a $D_h$-geodesic. 
Is it possible to compute the Hausdorff dimensions of $P\cap T_h^\alpha$ for each $\alpha \in [-2,2]$ with respect to the $D_h$ (resp.\ the Euclidean metric)? 
More weakly, as there a unique value of $\alpha$ which maximizes this dimension? 
In other words, is there a ``typical" thickness for a point on an LQG geodesic? 
\end{prob}

It is known that the Hausdorff dimensions considered in Problem~\ref{prob-geodesic} are a.s.\ equal to deterministic constants, see~\cite[Remark 1.12]{lqg-zero-one}. 
The analog of Problem~\ref{prob-geodesic} for a subcritical LQG metric ball boundary has been solved in~\cite{dg-lqg-dim,lqg-zero-one}. In that case, the maximizing value of $\alpha$ with respect to the Euclidean (resp.\ LQG) metric is $\alpha = \xi$ (resp.\ $\alpha = \gamma$).
One can also ask the analog of Problem~\ref{prob-geodesic} with Minkowski dimension instead of Hausdorff dimension. We expect that the answers will be the same.

\bibliography{cibib}
\bibliographystyle{hmralphaabbrv}

\end{document}